\documentclass[11pt]{amsart}

\usepackage[T1]{fontenc}
\usepackage{amsmath}
\usepackage{amsthm}
\usepackage{amsfonts}
\usepackage{amssymb}
\usepackage{yfonts}
\usepackage[all,cmtip]{xy}
\usepackage{amsxtra}
\usepackage{appendix}
\usepackage{tensor}
\usepackage{comment}

\usepackage[usenames,dvipsnames]{color}

\usepackage[pdftex, breaklinks, linktocpage=true, bookmarksopen=true,bookmarksopenlevel=0,bookmarksnumbered=true]{hyperref}

\RequirePackage{color}
\definecolor{bwgreen}{rgb}{0.183,1,0.5}
\definecolor{bwmagenta}{rgb}{0.7,0.0,0.1}
\definecolor{bwblue}{rgb}{0.317,0.161,1}

\hypersetup{
pdftitle = {Kisin and Wach modules via crystalline cohomology},
pdfauthor = {\textcopyright\ Bryden Cais},
pdfcreator = {\LaTeX\ with package \flqq hyperref\frqq},
colorlinks = {true},
linkcolor={bwmagenta},	 
anchorcolor={},	 
citecolor={bwblue},	
filecolor={},	
menucolor={},	
runcolor={},	
urlcolor={bwgreen}, 
}

\CompileMatrices

\DeclareFontFamily{OT1}{rsfs}{}
\DeclareFontShape{OT1}{rsfs}{n}{it}{<-> rsfs10}{}
\DeclareMathAlphabet{\mathscr}{OT1}{rsfs}{n}{it}

\DeclareFontFamily{OT1}{pzc}{}
\DeclareFontShape{OT1}{pzc}{n}{it}{<->s*[2.2]pzc}{}
\DeclareMathAlphabet{\mathpzc}{OT1}{pzc}{b}{sl}

\makeatletter

\newcommand{\Rmnum}[1]{\expandafter\@slowromancap\romannumeral #1@}
\makeatother

\DeclareMathOperator{\id}{id}

\DeclareMathOperator{\Frac}{Frac}

\DeclareMathOperator{\Hom}{Hom}
\DeclareMathOperator{\Ker}{Ker}

\DeclareMathOperator{\Gal}{Gal}

\DeclareMathOperator{\Spec}{Spec}

\DeclareMathOperator{\et}{\acute{e}t}

\DeclareMathOperator{\dR}{dR}
\DeclareMathOperator{\Cris}{Cris}
\DeclareMathOperator{\cris}{cris}

\DeclareMathOperator{\Lie}{Lie}

\DeclareMathOperator{\tors}{tors}
\DeclareMathOperator{\free}{fr}

\DeclareMathOperator{\Mod}{Mod}

\DeclareMathOperator{\Fil}{Fil}

\DeclareMathOperator{\st}{st}

\DeclareMathOperator{\tor}{tor}

\DeclareMathOperator{\im}{im}

\DeclareMathOperator{\Tor}{Tor}

\DeclareMathOperator{\pdiv}{pdiv}

\DeclareMathOperator{\PD}{PD}

\renewcommand*{\c}{\ensuremath{\mathbf{C}}}
\newcommand*{\Z}{\ensuremath{\mathbf{Z}}}
\newcommand*{\Q}{\ensuremath{\mathbf{Q}}}

\newcommand*{\Kbar}{\overline{K}}

\newcommand*{\m}{\mathfrak{M}}

\newcommand*{\s}{\mathfrak{S}}

\newcommand*{\C}{\mathbf{C}}

\newcommand*{\scrS}{\mathscr{S}}

\newcommand*{\scrG}{\mathscr{G}}

\newcommand*{\J}{\mathcal{J}}

\newcommand*{\scrM}{\mathscr{M}}

\renewcommand*{\O}{\mathscr{O}}

\newcommand*{\X}{\mathcal{X}}
\newcommand*{\scrX}{\mathscr{X}}
\newcommand*{\Y}{\mathcal{Y}}
\newcommand*{\RG}{R\Gamma}

\newcommand*{\scrD}{\mathscr{D}}

\newcommand*{\D}{\ensuremath{\mathbf{D}}}

\renewcommand*{\int}{\ensuremath{\mathrm{int}}}

\newcommand*{\e}{\ensuremath{\mathfrak{e}}}

\renewcommand*{\u}[1]{\underline{#1}}
\renewcommand*{\o}[1]{\overline{#1}}

\newcommand*{\wt}[1]{\widetilde{#1}}

\newcommand*{\inj}{\hookrightarrow}
\newcommand*{\onto}{\twoheadrightarrow}

\renewcommand{\tilde}{\widetilde}
\renewcommand{\bar}{\overline}

\DeclareMathOperator{\BT}{BT}

\theoremstyle{plain}
	
  \newtheorem{theorem}{Theorem}
   
  \newtheorem{proposition}[theorem]{Proposition}
  \newtheorem{lemma}[theorem]{Lemma}
  \newtheorem{corollary}[theorem]{Corollary}

\theoremstyle{definition}
  \newtheorem{definition}[theorem]{Definition}

\theoremstyle{remark}
  \newtheorem{example}[theorem]{Example}
  \newtheorem{remark}[theorem]{Remark}

\include{header}

\numberwithin{theorem}{section}
\numberwithin{equation}{section}

\date{First version December 4, 2014; Revised version \today.}

\begin{document}
\bibliographystyle{plain}

\title{Breuil--Kisin Modules via Crystalline cohomology}

\author{Bryden Cais}

\author{Tong Liu}

\thanks{This project began when the first author visited Christopher Davis
and Lars Hesselholt
in Copenhagen, and we are grateful for the hospitality
provided by the University of Copenhagen and the
many stimulating discussions that occurred there.
We also thank Bhargav Bhatt, Kiran Kedlaya, Peter Scholze, and Teruhisa Koshikawa  for their
helpful input on the project, and Yu Min for bringing an error in a previous version of this paper
 to our attention.}
\thanks{The first author is partially supported by a Simons Foundation Collaboration Grant.
The second author is partially supported by NSF grant DMS-1406926.}

\subjclass[2010]{Primary: 14F30 
Secondary: 11F80} 

\keywords{Breuil--Kisin Modules, crystalline cohomology}

\begin{abstract}
	For a perfect field $k$ of characteristic $p>0$
	and a smooth and proper formal scheme $\X$ over the ring of integers
	of a finite and totally ramified extension $K$ of $W(k)[1/p]$,
	we propose a cohomological construction of the Breuil--Kisin module attached
	to the $p$-adic \'etale cohomology $H^i_{\et}(X_{\o{K}},\Z_p)$.
	We then prove that our proposal works when $p>2$, $i < p-1$,
	and the crystalline cohomology of the special fiber of $\X$ is
	torsion-free in degrees $i$ and $i+1$.
\end{abstract}

\maketitle


\section{Introduction}\label{Intro}

Let $k$ be a perfect field of characteristic $p>0$ and $K$ a totally
ramified extension of degree $e$ over $W(k)[1/p]$. Fix an algebraic closure $\o{K}$ of $K$
and denote by $\c_K$ its $p$-adic completion.
  If $\X$  is a smooth proper formal scheme over $\O_K$
with (rigid analytic) generic fiber $X$, then the (torsion free part of the) $p$-adic \'etale cohomology
$T^i:=H^i_{\et}(X_{\Kbar},\Z_p)/\tors$ is a $G_K:=\Gal(\o{K}/K)$-stable lattice
in a crystalline representation. Functorially associated to
the $\Z_p$-linear dual $(T^i)^{\vee}$ is its Breuil--Kisin module $M^i$ over $\s:=W(k)[\![u]\!]$ in the sense
of\footnote{In fact, we work with a slightly different normalization than \cite{KisinFcrystal},
which is more closely related to (crystalline) cohomology; see Definition \ref{def:BKvar}
and Remark \ref{ClassicalKisin} for details.} \cite{KisinFcrystal}.  It is natural to ask
for a {\em direct cohomological} construction of $M^i$.

A cohomological construction of $M^i$ ``up to $p$-isogeny'' can be found in the Ph.D. thesis \cite{Bar} of B\"ar.
Writing $\O$ for the ring of rigid analytic functions on the open unit disk over $W(k)[1/p]$,
B\"ar constructs a perfect complex $\scrM(\X)$ of
sheaves of $\varphi$-modules over $\O$ on the special fiber $\X_k$
and a natural isomorphism of $\varphi$-modules
$H^i(\X_k,\scrM(\X)) \simeq M^i \otimes_{\s} \O$, at least when $\X$ is a scheme.
As scalar extension along $\s\rightarrow \O$
induces an equivalence between the isogeny category of Breuil--Kisin modules over $\s$
and the category of $\varphi$-modules over $\O$ that are pure of slope zero \cite[Lemma 1.3.13]{KisinFcrystal},
B\"ar's construction can indeed be viewed as providing a cohomological description of $M^i$ up to $p$-isogeny.
Unfortunately, extracting $M^i$ (up to $p$-isogeny) from $M^i\otimes_{\s} \O$
is rather indirect
({\em cf.} the proof of {\em loc.~cit.}),
so B\"ar's work does not yield as explicit a construction
as one would like.

More recently, the work of Bhatt, Morrow, and Scholze \cite{SBM}
associates to any smooth and proper formal scheme $\scrX$ over $\O_{\c_K}$
a perfect complex of $A_{\inf}$-modules $\RG_{A_{\inf}}(\scrX)$
whose cohomology groups are Breuil--Kisin--Fargues modules in the sense of \cite[Definition 4.22]{SBM}
(see also Definition \ref{BKFDef} below), and which is an avatar of all $p$-integral $p$-adic cohomology groups of $\scrX$.
One can deduce from their theory that if $\scrX=\X\times_{\O_K} \O_{\c_K}$ is defined
over $\O_K$,
then the base change $M^i_{A_{\inf}}:=M^i \otimes_{\s} A_{\inf}$
is a Breuil--Kisin--Fargues module, and one has a canonical identification
$H^i(\RG_{A_{\inf}}(\scrX)) \simeq M^i_{A_{\inf}}$
under the assumption that $H^i_{\cris}(\X_k/W(k))$ is torsion-free.
Note that with this assumption, $H^i_{\et}(X_{\o{K}},\Z_p)$ is also torsion free
by Theorem 14.5 and Proposition 4.34 of \cite{SBM}.
However,
this beautiful cohomological description of $M^i \otimes_{\s} A_{\inf}$
does not yield a cohomological
interpretation of the Breuil--Kisin module $M^i$ over $\s$, but only of its scalar extension
to $A_{\inf}$.

In this paper, assuming that $H^j_{\cris}(\X_k/W(k))$ is torsion-free for $j=i, i+1$,
we will provide a direct, cohomological construction of $M^i$
over $\s$, at least when $i < p-1$.  To describe our construction, we must first
introduce some notation.

Fix a uniformizer $\pi_0$ of $\O_K$, and let $E\in \s$ be the
minimal polynomial of $\pi_0$ over $W(k)$, normalized to have constant term $p$.
For each $n\ge 1$ choose $\pi_n \in \O_{\o{K}}$ satisfying $\pi_n^p=\pi_{n-1}$
and define $K_n:=K(\pi_n)$ and $K_{\infty}:=\cup_n K_n$.
For $n\ge 0$ we define $\s_n:=W(k)[\![u_n]\!]$, equipped with the unique continuous
Frobenius endomorphism $\varphi$ that acts on $W(k)$ as the unique lift $\sigma$
of the $p$-power map on $k$ and satisfies $\varphi(u_n)=u_n^p$.
We write
$\theta_n: \s_n\twoheadrightarrow \O_{K_n}$ for the continuous
$W(k)$-algebra surjection carrying $u_n$ to $\pi_n$, and we view $\s_n$ as a subring of $\s_{n+1}$
by identifying $u_n=\varphi(u_{n+1})$; this is compatible (via the $\theta_n$)
with the canonical inclusions $K_n\hookrightarrow K_{n+1}$.
We then see that $\varphi:\s_{n+1}\rightarrow \s_n$ is a ($\sigma$-semilinear)
isomorphism, so for $n\ge 1$ the element
\begin{equation}
	z_n:=E \varphi^{-1}(E)\cdots \varphi^{1-n}(E) = E(u_0)E^{\sigma^{-1}}(u_1)E^{\sigma^{-2}}(u_2)\cdots E^{\sigma^{1-n}}(u_{n-1})
	\label{Def:z_n}
\end{equation}
makes sense in $\s_n$.
Defining $z_0:=1$, we then have $\varphi(z_n)=\varphi(E)z_{n-1}$
for $n\ge 1$.

Write $S_n$ for the $p$-adic completion of the PD-envelope of $\theta_n$,
equipped with the $p$-adic topology.
This is naturally a PD-thickening of $\O_{K_n}$, equipped with
a descending filtration $\{\Fil^i S_n\}_{i\ge 0}$
obtained by taking the closure in $S_n$ of the usual PD-filtration.
The inclusions $\s_n\hookrightarrow \s_{n+1}$
uniquely extend to $S_n\hookrightarrow S_{n+1}$, and we henceforth consider $S_{n}$
as a subring of $S_{n+1}$ in this way.
Note that $\varphi$ uniquely extends to a continuous endomorphism $\varphi:S_n\rightarrow S_n$
which has image contained in $S_{n-1}$ (see Lemma \ref{KeyA}).
We identify $u_0=u$ and $\s_0=\s$, and will frequently write $S:=S_0$.

Given a smooth and proper formal scheme $\X$ over $\O_K$, we write $\X_{n}:=\X\times_{\O_K} \O_{K_n}/(p)$
for the base change to $\O_{K_n}/(p)$.  As $S_n\twoheadrightarrow \O_{K_n}/(p)$
is a divided power thickening, we can then form the crystalline
cohomology $\scrM^i_n:=H^i_{\cris}(\X_n/S_n)$ of $\X_n$ relative to $S_n$.
It is naturally a finite-type $S_n$-module with a $\varphi$-semilinear
endomorphism $\varphi_{\scrM}:\scrM^i_n\rightarrow \scrM^i_n$
and a descending and exhaustive filtration $\Fil^j \scrM^i_n:=\scrM^i_n \cap \Fil^j (\scrM^i_n[1/p])$ induced by the usual Hodge filtration
on de~Rham cohomology via the canonical comparison isomorphisms
$$
\scrM^i_n [1/p] \simeq S_n [1/p] \otimes_{W(k)} H^i_{\cris}(\X_{ k }/W(k))
\quad\text{and}\quad
K \otimes_{W(k)}  H^i_{\cris}(\X_{k }/W(k)) \simeq H^ i_{\dR}(X/ K).
$$
We refer to \S \ref{pdivGone} and in particular the discussion above Theorem \ref{thm-main-crystokisin} for details of this construction, and to \S\ref{subsec-1} for a more geometric  interpretation of this filtration via the crystalline cohomology of the divided powers of the canonical PD-ideal.
Give the localization $S_n[z_n^{-1}]$ the $\Z$-filtration by powers of $z_n$,
and equip $\scrM_n^i[z_n^{-1}]=\scrM_n\otimes_{S_n} S_n[z_n^{-1}]$ with the
tensor product filtration; that is, $\Fil^j(\scrM_n^i[z_n^{-1}]) := \sum_{a+b=j} z_n^a \Fil^b \scrM_n^i$,
with the sum ranging over all integers $a,b$ and taking place inside $\scrM_n^i[z_n^{-1}]$.
We then define
\begin{align*}
	\u{M}^i(\X) &= \varprojlim_{\varphi,n} \Fil^0(H^i_{\cris}(\X_n/S_n)[z_n^{-1}])\\
	&:=\{ \{\xi_n\}_{n\ge 0}\ :\ \xi_n\in \Fil^0(\scrM_n^i\otimes_{S_n} S_n[z_n^{-1}])\
	\text{and}\ (\varphi_{\scrM}\otimes\varphi)(\xi_{n+1})=\xi_{n}\ \text{for all}\ n\ge 0\}.	
\end{align*}
We equip $\u{M}^i(\X)$ with the Frobenius map
$\varphi_{\u{M}}(\{\xi_n\}):=\{(\varphi_{\scrM}\otimes \varphi)(\xi_n)\}_n$ and define
$$\Fil^j \u{M}^i(\X):=\{\{\xi_n\}_n \in \u{M}^i(\X)\ :\ \xi_0\in \Fil^j(\scrM_0^i)\}.$$
We view $\u{M}^i(\X)$ as an $\s=\s_0$-module
by $g(u)\cdot\{\xi_n\}_{n\ge 0}:=\{g^{\sigma^{-n}}(u_n)\cdot \xi_n\}_{n\ge 0}$.

We can now state our main result, which provides
a cohomological description of Breuil--Kisin modules in Hodge--Tate weights at most $p-2$:

\begin{theorem}\label{MainIntro}
	Assume that $p>2$.
	Let $\X$ be a smooth and proper formal scheme
	over $\O_K$ and $i$ an integer with $0\le i < p-1$,
	and let $M$ be the Breuil--Kisin module associated to
	the $\Z_p$-dual of the Galois lattice $H^i_{\et}(X_{\o{K}},\Z_p)/\tors$.
	If $H^j_{\cris}(\X_k/W(k))$ is torsion-free for $j=i, i+1$, then there is a natural
	isomorphism of Breuil--Kisin modules
	\begin{equation*}
		M \simeq \u{M}^i(\X).
	\end{equation*}
\end{theorem}

A few remarks are in order regarding the hypotheses of Theorem \ref{MainIntro}.
The assumption that $0 \le i < p-1$ is essential since for $i \ge p-1$
the category of height-$i$ Breuil modules (the structure with which $H^i_{\cris} (\X_0/S)$ has been endowed)
is {\em not} equivalent to the category of height $i$ Breuil--Kisin modules over $\s$.
Furthermore, to get $p$-integral comparison isomorphisms, we expect there must always some torsion freeness assumptions;
see, for example, \cite[\S2]{SBM}.
On the other hand, it may be possible to weaken the assumption that $H^j_{\cris}(\X_k/W(k))$ is torsion-free for $j=i, i+1$. For example, the assumption that $H^{i+1}_{\cris}(\X_k/W(k))$ is $p$-torsion free can be replaced by
assuming instead that $H^{i+1}(\RG_{A_{\inf}}(\scrX))$ is $u$-torsion free; see Proposition \ref{prop-isotiliota}.

The proof of Theorem \ref{MainIntro} has two major---and fairly independent---ingredients,
one of which might be described as purely cohomological, and the other as purely (semi)linear algebraic.
Fix a nonnegative integer $r < p-1$, and let $\Mod_S^{\varphi,r}$ be the category of
{\em height-$r$ quasi-Breuil modules over $S$}, whose objects are triples $(\scrM,\Fil^r\scrM,\varphi_{\scrM,r})$
where $\scrM$ is a finite, free $S$-module, $\Fil^r\scrM\subseteq \scrM$ is a submodule containing
$(\Fil^rS)\scrM$ with the property that $\scrM/\Fil^r\scrM$ is $p$-torsion free,
and $\varphi_{\scrM,r}:\Fil^r\scrM\rightarrow\scrM$ is a $\varphi$-semilinear map
whose image generates $\scrM$ as an $S$-module.  Morphisms are filtration and $\varphi$-comaptible
$S$-module homomorphisms.  For each $j\in \Z$, we then define $S$-submodules $\Fil^j\scrM:=\{m\in \scrM\ :\ E^{r-j}m\in \Fil^r\scrM\}$ for $j\le r$ and we put $\Fil^j\scrM=0$ for $j>r$.
We similarly define the category $\Mod_{\s}^{\varphi,r}$ of {\em height-$r$ filtered Breuil--Kisin modules},
whose objects are triples $(M,\Fil^rM,\varphi_{M,r})$ where $M$ is a finite and free $\s$-module,
$\Fil^rM\subseteq M$ is a submodule containing $E^rM$ with $M/\Fil^rM$ having no $p$-torsion,
and $\varphi_{M,r}:\Fil^r M\rightarrow M$ is a $\varphi$-semilinear map whose image
generates $M$ as an $\s$-module, and we define $\Fil^j M:=\{m\in M\ :\ E^{r-j}m\in \Fil^rM\}$
 for $j\le r$, with $\Fil^j M=0$ when $j >r$.
It is well-known that $\Mod_{\s}^{\varphi,r}$ is equivalent to the ``usual" category
of Breuil--Kisin modules over $\s$; see Remark \ref{ClassicalKisin}.
Scalar extension along $\s\rightarrow S$ induces a covariant functor
$\u{\scrM}:\Mod_{\s}^{\varphi,r}\rightarrow \Mod_{S}^{\varphi,r}$ which is known
to be an equivalence of categories \cite[Theorem 2.2.1]{CarusoLiu}.
Our main ``(semi)linear-algebraic" result is that the functor
$\u{M}:\Mod_S^{\varphi,r}\rightarrow \Mod_{\s}^{\varphi,r}$
defined by $\u{M}(\scrM):=\varprojlim_{\varphi,n} \Fil^0(\scrM\otimes_S S_n[z_n^{-1}])$
is a quasi-inverse to $\u{\scrM}$.  This we establish using a structural result (Lemma \ref{specialbasis})
that provides an explicit description of a Breuil module via bases and matrices,
together with a sequence of somewhat delicate Lemmas that rely on the
fine properties of the rings $S_n$ and their endomorphisms $\varphi$.

On the other hand, if $\X$ is a smooth and proper formal $\O_K$-scheme and $i\le r < p-1$, then
the crystalline cohomology $\scrM:=H^i_{\cris}(\X_0/S)$
can be naturally promoted to an object of $\Mod_S^{\varphi,i}$.
Using the results of Bhatt, Morrow, and Scholze \cite{SBM},
when $H^j_{\cris}(\X_k/W(k))$ is torsion free for $j=i, i+1$,
we prove in \S\ref{pdivGone} that one has a canonical comparison isomorphism
$$\Hom_{S,\Fil,\varphi}(\scrM,A_{\cris})\simeq H^i_{\et}(X_{\o{K}},\Z_p),$$
from which we deduce that $\u{M}(\scrM)$ may be identified
with the (filtered) Breuil--Kisin module $M^i$ attached to
the $\Z_p$-linear dual of $H^i_{\et}(X_{\o{K}},\Z_p)/\tors$.
Theorem \ref{MainIntro} then follows.

\section{Ring-theoretic constructions}

We keep the notation of \S\ref{Intro}. Note that, by the very definition,
the ring $S_n$ is topologically generated as an $\s_n$-algebra
by the divided powers $\{E^i/i!\}_{i\ge 1}$.  It follows that
$\Fil^i S_n$ is the closure of the expanded ideal  $(\Fil^i S) S_n$ in $S_n$.
We write $c_0:=\varphi(E)/p$, which is a unit of $S=S_0$.
Since $\varphi(g)\equiv g^p\bmod p$, one shows that $c_0 = v + E^p/p$ for a unit $v\in \s$.
Observe that
$$\left(\frac{E^p}{p}\right)^n = \frac{(pn)!}{p^n}\left(\frac{E^{pn}}{(pn)!}\right),\ \text{and}\
v_p((pn)!/p^n) = v_p(n!)
$$
by Legendre's formula, so that the ring
$T_n:=\s_{n}[\![E^p/p]\!]$ is naturally a subring of $S_n$ that contains $c_0$
and is stable under $\varphi$ as $\varphi(E)=pc_0$.  There are obvious inclusions $T_n\hookrightarrow T_{n+1}$
that are compatible with the given inclusions $\s_n\hookrightarrow \s_{n+1}$
and $S_n\hookrightarrow S_{n+1}$.
By definition, the injective map $\varphi:\s_n\rightarrow \s_n$
has image precisely $\s_{n-1}$ inside $\s_n$.  While the na\"ive analogue of this fact
for the rings $S_n$ is certainly {\em false},
Frobenius is nonetheless a ``contraction" on $S_n$ in the following precise sense:

\begin{lemma}\label{KeyA}
	Let $i$ be a nonnegative integer and set $b(i):=\left\lceil i\frac{p-2}{p-1} \right\rceil$.
	We have $\varphi(S_n) \subseteq T_{n-1}$ inside $S_n$; in particular, $\varphi:S_n\rightarrow S_n$
	has image contained in $S_{n-1}$.  Moreover,
	if $x\in \Fil^i S_n$ then $\varphi(x)=w+y$ for some $w\in \s_{n-1}$ and $y\in \Fil^{pb(i)} S_{n-1}$.
\end{lemma}

\begin{proof}
	Since $\Fil^i S_n$ is topologically generated as an $\s_n$-module by $\{E^j/j!\}_{j\ge i}$
	and $\varphi(\s_n)=\s_{n-1}$, to prove the first assertion it is enough to show that $\varphi(E^j/j!)$
	lies in $T_{n-1}$ for all $j$.  But this is clear, as $\varphi(E^j/j!) = c_0^j p^j/j!$
	and $p^j/j!\in \Z_p$ for all $j$.  To prove the second assertion, it likewise suffices to
	treat only the cases $x=E^j/j!$ for $j\ge i$.
	As observed above, $\varphi(E)=E^p + pv$ for $v\in \s_0^{\times}$, so we compute
	\begin{equation}
		\varphi(E^j/j!) = \frac{(E^p + pv)^j}{j!} = \sum_{k=0}^j \frac{p^k}{(j-k)!k!} E^{p(j-k)} v^k.
		\label{binomexp}
	\end{equation}
	Writing $s_p(n)$ for the sum of the p-adic digits of any nonnegative integer $n$
	and again invoking Legendre's formula gives
	$$v_p\left(\frac{p^k}{(j-k)!k!}\right) = k - \frac{j}{p-1} + \frac{s_p(j-k)+s_p(k)}{p-1},$$
	which is nonnegative for $k \ge j/(p-1)$.  On the other hand, if $k < j/(p-1)$
	then one has the inequality $(j-k) \ge \left\lceil j \frac{p-2}{p-1}\right\rceil = b(j)$.
	Combining these observations with (\ref{binomexp}) then gives the desired
	decomposition $\varphi(E^j/j!) = w+y$ with $w\in \s_0$ the sum of all terms
	in (\ref{binomexp}) with $k \ge j/(p-1)$ and $y\in \Fil^{pb(j)} S_{n-1}$
	the sum of the remaining terms.
\end{proof}

We now define
$$
\varprojlim_{\varphi,n} S_n :=
	 \left\{\{s_n\}_{n\ge 0} \ :\ s_n \in S_n\ \text{and}\ \varphi(s_{n+1})=s_{n},\ \text{for all}\ n\ge 0\right\},
$$
which---as $\varphi$ is a ring homomorphism---has the natural structure of a ring via component-wise addition and multiplication.
The fact that $\varphi$ ``contracts" the tower of rings $\{S_n\}_{n\ge 0}$
manifests itself in the following Lemma, which inspired this paper:

\begin{lemma}\label{FrobComp2}
	For $p>2$ the natural map
	\begin{equation}
		\s \rightarrow \varprojlim_{\varphi,n} S_n,\qquad g(u)\mapsto \{g^{\sigma^{-n}}(u_n)\}_{n\ge 0}
		\label{FrakStoLimS}
	\end{equation}
	is an isomorphism of rings.
\end{lemma}

\begin{proof}
	It is clear that the given map is an injective ring homomorphism, so it suffices to prove that
	it is surjective.  Let $\{s_n\}_{n\ge 0}$ be an arbitrary element of $\varprojlim_{\varphi,n} S_n$.
	Since $S_n = \s_n+\Fil^p S_n$, an easy induction using Lemma \ref{KeyA} shows
	that $s_0=\varphi^{(n)}(s_n)$ lies in  $\s_0 + \Fil^{i_n} S_0$, where $i_n$ is defined recursively by $i_0=p$
	and $i_{n}=p b(i_{n-1})$ for $n\ge 1$.  As this holds for all $n\ge 0$ and
	$$i_{n+1}-i_n = p b(i_n) - i_n \ge pi_n\frac{p-2}{p-1} -i_n = (p-3)i_n+\frac{p-2}{p-1}i_n$$
	so that $\{i_n\}_{n\ge 0}$ is an increasing sequence (recall $p>2$), it follows
	that $s_0\in \s_0$.  But then $\{s_n\}_{n\ge 0} = \{\varphi^{-n}(s_0)\}_{n\ge 0}$
	is in the image of (\ref{FrakStoLimS}), as desired.
\end{proof}

\begin{remark}
	An equivalent formulation of Lemma \ref{FrobComp2} is stated in passing in the introduction to \cite{Bar},
	which does not ultimately use the rings $S_n$.
	We first learned about Lemma \ref{FrobComp2} from \cite{CaisDavis},
	which constructs a canonical isomorphism of rings $\s\simeq \varprojlim W(\O_{K_n})$,
	where $W(\cdot)$ is the functor of ($p$-typical) Witt vectors (see \cite[Theorem 1.4]{CaisDavis})
	and the inverse limit is taken along the Witt vector Frobenius mappings.
	As the kernel of the projection $W(\O_{K_n})\twoheadrightarrow \O_{K_n}$
	on to the 0-{\em th} Witt component
	is equipped with divided powers, the $W$-algebra map $\s\rightarrow W(\O_{K_n})$
	sending $u$ to the Teichm\"uller lifting $[\pi_n]$ extends to an inclusion $S_n\hookrightarrow W(\O_{K_n})$.
	Since $\varprojlim$ is left exact, this gives an alternate proof of Lemma \ref{FrobComp2}.
\end{remark}

For later use, we record here the following elementary result:

\begin{lemma}\label{Lem:intersect}
	Let $n$ and $m$ be any nonnegative integers.
	Then
	\begin{enumerate}
		\item $\Fil^m S_n \cap \s_n = E^m \s_n$ inside $S_n$.
		\item $(\Fil^m S_n)[1/p] \cap S_n = \Fil^m S_n$ inside $S_n[1/p]$.
	\end{enumerate}
\end{lemma}

\begin{proof}
	We must prove that $\s_n/E^m\s_n\rightarrow S_n/\Fil^m S_n$ is injective
	with target that is $p$-torsion free.  This is an easy induction on $m$,
	using the fact that $(E^m)/(E^{m+1})$ and $\Fil^m S_n/\Fil^{m+1} S_n$
	are free of rank one over $\s_n/(E) \simeq S_n/\Fil^1 S_n \simeq \O_{K_n}$
	with generators $E^m$ and $E^m/m!$, respectively.
\end{proof}

\section{Breuil and Breuil--Kisin modules}\label{BTconst}

We begin by recalling the relation between Breuil--Kisin modules
and Breuil modules in low Hodge--Tate weights.
Throughout, we fix an integer $r < p-1.$

\begin{definition}\label{def:BKvar}
	We write $\Mod_{\s}^{\varphi,r}$ for the category of {\em height-$r$ filtered Breuil--Kisin modules} over $\s$
	whose objects are triples $(M,\Fil^r M, \varphi_{M,r})$ where:
	\begin{itemize}
		\item $M$ is a finite free $\s$-module,
		\item $\Fil^r M$ is a submodule with $E^r M\subseteq \Fil^r M$ and $M/\Fil^r M$  $p$-torsion free.
		\item $\varphi_{M,r}:\Fil^r M\rightarrow M$ is a $\varphi$-semilinear map
		whose image generates $M$ as an $\s$-module.
	\end{itemize}
		Morphisms are $\s$-module homomorphisms which are compatible with the additional structures.
	For any object $(M,\Fil^r M, \varphi_{M,r})$ of $\Mod_{\s}^{\varphi,r}$
	and $i\le r$, set
	\begin{equation}
		\Fil^i M := \left\{ m \in M\ :\ E^{r-i}m \in \Fil^r M  \right\}\label{FilDefnM}
	\end{equation}
	and put $\Fil^i M:=0$ for $i>r$,
	and define a $\varphi$-semilinear map $\varphi_M:M\rightarrow M$
	by 
	\begin{equation}\label{phiDefonM}
		\varphi_M(m): = \varphi_{M,r}(E^r m)
	\end{equation} for $m\in M$.
	Note that for $m\in \Fil^r M$ we have $\varphi_M(m) = \varphi(E)^r\varphi_{M,r}(m)$.
\end{definition}

\begin{remark}\label{ClassicalKisin}
	Our definition of the category $\Mod_{\s}^{\varphi,r}$
	is perhaps non-standard ({\em cf.}~\cite{CarusoLiu}).
	In the literature, one usually works instead with the category of
	{\em Breuil--Kisin modules} (without filtration), whose objects are
	pairs $(\m,\varphi_{\m})$ consisting of a finite free $\s$-module
	$\m$ and a $\varphi$-semilinear map $\varphi_{\m}:\m\rightarrow \m$
	whose linearization has cokernel killed by $E^r$, with evident morphisms.
	However, the assignment $(M,\Fil^r M, \varphi_{M,r})\rightsquigarrow (\Fil^r M , E^r\varphi_{M,r})$
	induces an equivalence between our category $\Mod_{\s}^{\varphi,r}$
	and the ``usual" category of Breuil--Kisin modules $(\m,\varphi_{\m})$.
	While this is fairly standard (e.g. \cite[Lemma 8.2]{Lau:Frames}
	or \cite[Lemma 1]{VZ}), for the convenience
	of the reader and for later reference, we describe a quasi-inverse.
	
	Given $(\m,\varphi_{\m})$ as above and
	writing $\wt{\varphi}:\varphi^* \m\rightarrow \m$ for the linearization
	of $\varphi$, there is a unique (necessarily injective) $\s$-linear map
	$$\psi:\m\rightarrow \varphi^*\m\quad\text{satisfying}\quad \wt{\varphi}\circ\psi=E^r\cdot\id.$$
	The corresponding filtered Breuil--Kisin module over $\s$ is then given by
	\begin{equation}
		M:=\varphi^*\m,\ \Fil^r M:=\psi(\m)\ \text{with}\
	 	\varphi_{M,r}(\psi(m)):=1\otimes m.\label{BK-dictionary}
	 \end{equation}
	Alternatively, as one checks easily, we have the description
	\begin{equation}
		\Fil^r M = \Fil^r \varphi^*\m =\{x\in \varphi^*\m\ :\ (1\otimes\varphi)(x)\in E^r\m\}.
		\label{BKFil-Alt}
	\end{equation}
	From (\ref{BK-dictionary}) it is clear that $M$  and $\Fil^r M$ are then {\em free} $\s$-modules, so that
	$M/\Fil^r M$ has projective dimension $1$ over $\s$.  It follows
	from the Auslander--Buchsbaum formula and Rees' theorem that $M/\Fil^r M$
	has depth $1$ as an $\s$-module, so since $\s$ has maximal ideal $(u,p)=(u,E)$
	and $E$ is a zero-divisor on $M/\Fil^r M$, it must be that $\pi_0=u\bmod E$
	is {\em not} a zero divisor on $M/\Fil^r M$ and hence this quotient is $p$-torsion
	free and we really do get a filtered Breuil--Kisin module in this way.

	We have chosen to work with our category $\Mod_{\s}^{\varphi,r}$ of filtered Breuil--Kisin modules
	instead of the ``usual" category of Breuil--Kisin modules as it is our category whose objects
	are inherently ``cohomological", as we shall see.
\end{remark}

Let $S=S_0$ be as above, 
and for $i\ge 1$
write $\Fil^i S\subseteq S$ for the (closure of the) ideal generated by $\{E^n/n!\}_{n\ge i}$.
For $i \le r$ One has $\varphi(\Fil^i S) \subseteq p^iS$,
so we may define $\varphi_i:\Fil^i S\rightarrow S$ as $\varphi_i:=p^{-i}\varphi$.

\begin{definition}
Denote by $\Mod_S^{\varphi,r}$ the category of {\em height-$r$ $($quasi$)$ Breuil-modules}
over $S$.  These are triples $(\scrM,\Fil^r \scrM, \varphi_{\scrM,r})$
consisting of a finite free $S$-module $\scrM$ with an $S$-submodule $\Fil^r\scrM$
and a $\varphi$-semilinear map $\varphi_{\scrM,r}:\Fil^r\scrM\rightarrow \scrM$
such that:
\begin{enumerate}
	\item $(\Fil^r S)\scrM\subseteq \Fil^r\scrM$ and $\scrM/\Fil^r\scrM$
	has no $p$-torsion.
	
	\item The image of $\varphi_{\scrM,r}$ generates $\scrM$ as an $S$-module
\end{enumerate}
Morphisms are $S$-module homomorphisms that are compatible with the additional structures.
Given a quasi Breuil module $(\scrM,\Fil^r \scrM,\varphi_{\scrM,r})$ of height $r$, for
$ i\le r$ we set
\begin{equation}
	\Fil^i \scrM := \left\{ m \in \scrM\ :\ E^{r-i}m\in \Fil^r \scrM \right\}
	\label{FilDefn}
\end{equation}
and we put $\Fil^i \scrM:=0$ for $i>r$ and define
$\varphi_{\scrM}:\scrM\rightarrow \scrM$ by the recipe
$$\varphi_{\scrM}(m) := c_0^{-r}\varphi_{\scrM,r}( E^r m)\quad\text{with}\quad c_0=\varphi(E)/p\in S^{\times}.$$
Note that $\varphi_{\scrM} = p^r\varphi_{\scrM,r}$ on $\Fil^r\scrM$; it follows that $\varphi_{\scrM}$
and $\varphi_{\scrM,r}$ determine each other.
\end{definition}

There is a canonical ``base change" functor
\begin{equation}
	\u{\scrM}:\Mod_{\s}^{\varphi,r}\rightarrow \Mod_{S}^{\varphi,r}\label{BC}
\end{equation}
defined as follows: if $(M,\Fil^r M, \varphi_{M,r})$ is an object of $\Mod_{\s}^{\varphi,r}$,
then we define
$\scrM:=S\otimes_{\s} M$ and $\varphi_{\scrM}:=\varphi\otimes \varphi_M$,
with $\Fil^r \scrM$ the submodule
generated by the images of $S\otimes_{\s} \Fil^r M$ and $\Fil^r S\otimes_{\s} M$.
Then by definition, the restriction of $\varphi_{\scrM}$ to $\Fil^r\scrM$
has image contained in $\varphi^r(E) \scrM$, so it makes sense
to define $\varphi_{\scrM,r}:=p^{-r}\varphi_{\scrM}$ on $\Fil^r\scrM$.
Using the definition of the category $\Mod_{\s}^{\varphi,r}$,
it is straightforward to check that this defines a covariant
functor from $\Mod_{\s}^{\varphi,r}$ to $\Mod_S^{\varphi,r}$.

\begin{remark}
	Let $(M,\Fil^r M, \varphi_{M,r})$ be any filtered Breuil--Kisin module over $\s$
	with associated Breuil module $(\scrM, \Fil^r\scrM, \varphi_{\scrM,r})$ over $S$.
	Writing $(\m,\varphi_{\m})$ for the ``classical" Breuil--Kisin module over $\s$
	given as in Remark \ref{ClassicalKisin} and $\varphi:\s\rightarrow S$ for the composition of inclusion with Frobenius,
	one checks using (\ref{BK-dictionary}) and (\ref{BKFil-Alt}) that we have $\scrM = S \otimes_{\varphi,\s} \m$ with
	\begin{equation*}
		\Fil^r \scrM = \{m \in \scrM = S\otimes_{\varphi,\s}\m \ :\
		(1\otimes \varphi_{\m})(m)\in \Fil^r S\otimes_{\s} \m \}
	\end{equation*}
	and $\varphi_{\scrM,r}$ is the composite
	\begin{equation*}
		\xymatrix{
			{\Fil^r \scrM} \ar[r]^-{1\otimes \varphi_{\m}} & {\Fil^r S \otimes_{\s}\m }
			\ar[r]^-{\varphi_r\otimes 1} & S\otimes_{\varphi,\s} \m = \scrM
			} .
	\end{equation*}
\end{remark}

It is known that the functor $(\ref{BC})$ is an equivalence of categories.
When $r=1$, this follows from work of Kisin \cite[2.2.7, A.6]{KisinFcrystal},
albeit in an indirect way as the argument passes
through Galois representations.
Caruso and Liu \cite{CarusoLiu} give a proof of this equivalence for general $r< p-1$
by appealing to the work of Breuil and using pure (semi)linear algebra with bases and matrices.
However, no existing proof provides what one could reasonably call a {\em direct} description of a quasi-inverse functor.
We will use the ideas of section \ref{pdivGone} to provide such a description.
Before doing so, however,
we work out an instructive example:

\begin{example}
	The (filtered) Breuil--Kisin module attached to Tate module of the $p$-divisible group
	$\mu_{p^{\infty}}$
	is the object of $\Mod_{\s}^{\varphi,1}$
	given by $M=\s\cdot \e$ on which Frobenius acts as $\varphi_M(\e)=\varphi(E)\cdot \e$,
	with $\Fil^1 M = M$ and $\varphi_{M,1}(\e)=\e$.
	The corresponding Breuil module $\scrM=S\cdot e$ is of rank 1 over $S$ with Frobenius
	acting as $\varphi_{\scrM}(e) = \varphi(E)\cdot e$ and we have $\Fil^1\scrM = \scrM$ with $\varphi_{\scrM,1}(e)=c_0 e$
	where $c_0=\varphi(E)/p\in S^{\times}$.
	Defining $\lambda:=\prod_{n\ge 0} \varphi^{n}(c_0)$,
	we have that $\lambda\in S^{\times}$ satisfies $\lambda/\varphi(\lambda) = c_0$.
	It follows that multiplication by $\lambda$ carries $\scrM$
	isomorphically onto the Breuil module given by the triple $(S,S,\varphi)$.
	
	Let $z_n\in \s_n$ be as in (\ref{Def:z_n}) and give $S_n[z_n^{-1}]$ the $\Z$-filtration by powers
	of $z_n$.  Define
	\begin{equation*}
		\u{M}(\scrM):=\{  \{\xi_n\}_{n\ge 0}\ :\ \xi_n \in \Fil^0(\scrM\otimes_{S_0} S_n[z_n^{-1}]),\ \text{and}\
		(\varphi\otimes \varphi)(\xi_n)=\xi_{n-1},\ n\ge 1\}
	\end{equation*}
	which we give the structure of an $\s=\s_0$-module by the rule
	$$g(u_0)\cdot \{\xi_n\}_{n\ge 0} := \{g^{\sigma^{-n}}(u_n)\xi_n\}_{n\ge 0},$$
	where each $\scrM\otimes_{S_0} S_n[z_n^{-1}]$ is viewed as a module over $\s_n$
	through the right factor and the canonical inclusion $\s_n\hookrightarrow S_n$.
	
	We then claim that the $\s$-linear map
	\begin{equation*}
		\iota:M=\s\cdot \e \rightarrow \u{M}(\scrM) \qquad\text{determined by}\qquad
		\iota(\e):=\{ e \otimes z_n^{-1}\}_{n\ge 0}
	\end{equation*}
	is an isomorphism.
	
	To see this, first note that the map is well defined
	as
	$$e\otimes z_n^{-1}\in \Fil^1 \scrM \otimes \Fil^{-1} S_n[z_n^{-1}]\subseteq \Fil^0(\scrM\otimes S_n[z_n^{-1}])$$
	and
	$$(\varphi\otimes\varphi)(e\otimes z_n^{-1})=\varphi(E)e \otimes (\varphi(E)z_{n-1})^{-1} = e\otimes z_{n-1}^{-1}$$
	for all $n\ge 1$ (recall that $z_0=1$).
	It is clear from the very construction that $\iota$ is injective.
	To see surjectivity, we just observe that every element of $\xi_n\in \Fil^0(\scrM\otimes_{S_0} S_n[z_n^{-1}])$
	may be written as a simple tensor $\xi_n=e\otimes s_n/z_n$ with $s_n\in S_n$.
	The condition that the $\xi_n$ form a $\varphi$-compatible sequence is
	then simply that $\varphi(s_n)=s_{n-1}$, {\em i.e.}~that $\{s_n\}_{n\ge 0}$
	lies in the projective limit $\varprojlim_{\varphi,n} S_n$, which is exactly the image of
	$\s_0$ under the natural map thanks to Lemma \ref{FrobComp2}.
	It follows immediately from this that $\{\xi_n\}_{n\ge 0}$ lies in the image of $\iota$,
	as desired.
\end{example}

\begin{remark}
	The intrepid reader may wish to work out the analogue of this example for
	the Tate module of the $p$-divisible group $\Q_p/\Z_p$, whose associated filtered Breuil--Kisin module
	is given by $M=\s\cdot \e$ with $\Fil^1 M = EM$ and $\varphi_{M,1}(E\cdot \e) = \e$.
	The corresponding Breuil module is
	given by the triple $(S,\Fil^1 S, \varphi_1)$.
	As it turns out, this computation is significantly more involved,
	and requires Lemma \ref{KeyC} (for $d=1$) to carry out successfully.
\end{remark}

With this motivating example, we may now formulate our main result,
which is an {\em explicit} description of a quasi-inverse to (\ref{BC}).
This allows us to realize Breuil--Kisin modules with Hodge--Tate weights
in $\{0,\ldots,p-2\}$ as ``Frobenius-completed cohomology
up the tower $\{K_n\}_n$."

\begin{definition}\label{Def:Mfunctor}
	For $(\scrM,\Fil^r\scrM,\varphi_{\scrM,r})$ any object of $\Mod_{S}^{\varphi,r}$,
	we define
	\begin{align*}
		\u{M}(\scrM)&:=\varprojlim_{\varphi,n} \Fil^0(\scrM \otimes_S S_n[z_n^{-1}])\\
		 &= \left\{\{\xi_n\}_{n\ge 0}\ :\ \xi_n\in \Fil^0(\scrM\otimes_S S_n[z_n^{-1}])\ \text{and}\
		 (\varphi_{\scrM}\otimes\varphi)(\xi_n)=\xi_{n-1}\ \text{for}\ n\ge 1\right\}
	\end{align*}
	with filtration
	\begin{align*}
	\Fil^i\u{M}(\scrM):=
		  \left\{\{\xi_n\}_{n\ge 0}\in \u{M}(\scrM)\ :\ \xi_0\in \Fil^i(\scrM)\otimes_S S_0
		 \right\}.
	\end{align*}
	We equip $\u{M}(\scrM)$ with the Frobenius $\varphi_{\u{M}}$ given by
	$$\varphi_{\u{M}}(\{\xi_n\}_{n\ge 0}):=\{(\varphi_{\scrM}\otimes \varphi)(\xi_n)\}_{n\ge 0}.$$
	and give $\u{M}(\scrM)$ the structure of an $\s$-module via
	$$g\cdot \{\xi_n\}_{n\ge 0}:=\{ g^{\sigma^{-n}}(u_n)\xi_n\}_{n\ge 0}\quad\text{for}\quad g\in \s=\s_0.$$
	It is straightforward to check that $\varphi_{\u{M}}$ is a $\varphi$-semilinear map on $\u{M}(\scrM)$.
\end{definition}

We will see in Corollary \ref{IsAFun}
that the functor $\u{M}$ so defined takes values in
$\Mod_{\s}^{\varphi,r}$, so in particular
the restriction of $\varphi_{\u{M}}$ to $\Fil^r \u{M}(\scrM)$
is divisible by $\varphi(E)^r$ and
$\varphi_{\u{M},r}:=\varphi(E)^{-r}\varphi_{\u{M}}$ makes sense on
$\Fil^r \u{M}(\scrM)$.

\begin{theorem}\label{MainThm}
	The construction $\scrM\rightsquigarrow \u{M}(\scrM)$
	defines a covariant functor
	$$\u{M}:\Mod_S^{\varphi,r}\rightarrow \Mod_{\s}^{\varphi,r}$$
	that is moreover a quasi-inverse to the functor $\u{\scrM}$
	of $(\ref{BC})$.
\end{theorem}

We will establish Theorem \ref{MainThm} through a sequence of lemmas.
We begin with a structural result for Breuil modules
which shows, in particular, that the functor (\ref{BC})
is essentially surjective:

\begin{lemma}\label{specialbasis}
	Let $\scrM \in \Mod_{S}^{\varphi,r}$.  There is an $S$-basis $e_1,\ldots, e_d$
	of $\scrM$ and matrices $A, B\in M_d(\s)$ such that:
	\begin{enumerate}
		\item If $(\alpha_1,\ldots,\alpha_d):=(e_1,\ldots, e_d)A$ then
		$$\Fil^r \scrM = \bigoplus_{i=1}^d S\alpha_i+ \Fil^p \scrM.$$\label{Areq}
		\item $c_0^{-r}\varphi_{\scrM,r}(\alpha_i)=e_i$ for $1\le i\le d$
		\item $(e_1,\ldots,e_d)\cdot E^r = (\alpha_1,\ldots,\alpha_d) B$\label{Breq}
		\item $\varphi_{\scrM}(e_1,\ldots,e_d) = (e_1,\ldots,e_d)\varphi(B)$ for $1\le i\le d$
		\item $AB=BA=E^r$.
	\end{enumerate}
	In particular, the $\s$-module $M:=\bigoplus_{i=1}^d \s e_i$
	with $\Fil^r M:=\bigoplus_{i=1}^d \s \alpha_i$ and $\varphi_{M,r}$
	determined by
	$\varphi_{M,r}(\alpha_i):=e_i$ is an object of $\Mod_{\s}^{\varphi,r}$
	whose image under $(\ref{BC})$ is $\scrM$.
\end{lemma}

\begin{proof}
	This is \cite[Lemma 2.2.2]{CarusoLiu}. 	
\end{proof}

\begin{remark}
	We emphasize that the proof of Lemma \ref{specialbasis} given in
	\cite{CarusoLiu}---which relies on (the easy part of) \cite[Lemma 4.1.1]{LiuT-CofBreuil}---uses only (semi)linear algebra.
	While this result establishes the essential
	surjectivity of the functor (\ref{BC}),
	the proof that this functor is an equivalence given in \cite[Theorem 2.2.1]{CarusoLiu}
	relies on (a generalization of) the full-faithfulness result \cite[1.1.11]{Kisin-Modularity},
	which uses certain auxilliary categories of torsion Breuil--Kisin and Breuil modules
	and a devissage argument to reduce to the $p$-torsion case, where the result is a consequence
	of (the proof of) \cite[3.3.2]{BreuilIntegral} using Lemma 2.1.2.1 and Proposition 2.1.2.2
	of \cite{Breuil} and the argument of \cite[Theorem 4.1.1]{Breuil-normes}.
	In contrast, by writing down an explicit quasi-inverse to (\ref{BC}),
	our proof of Theorem \ref{MainThm} uses neither devissage nor any auxilliary
	categories, and in particular does not rely on \cite{Breuil-normes}, \cite{Breuil},
	\cite{BreuilIntegral}, or \cite{Kisin-Modularity}.
\end{remark}

In what follows, given an object $\scrM$ of $\Mod_S^{\varphi,r}$,
an $S$-basis $e_1,\ldots, e_d$ of $\scrM$, and an $S$-algebra $S'$,
we will abuse notation slightly and again write $e_1,\ldots, e_d$ for the induced
$S'$-basis of $\scrM\otimes_S S'$.

\begin{lemma}\label{alpharep}
	Let $\scrM\in \Mod_{S}^{\varphi,r}$, and let $A$
	be as in Lemma $\ref{specialbasis}$.  For $n\ge 1$, any element $\xi_n$
	of $\Fil^0( \scrM\otimes_{S} S_n[z_n^{-1}])$ may be expressed
	in the form
	$$\xi_n =z_n^{-r}(e_1,\ldots,e_d)\cdot (Ax_n+y_n)$$
	with $x_n$ a $($column$)$ vector in $S_n^d$ and $y_n$ a vector
	in $(\Fil^p S_n)^d$.
\end{lemma}

\begin{proof}
	Assume $n\ge 1$ and
	observe first that for $0\le i \le r$ we have the containment
	$$\Fil^i \scrM \otimes_{S} \Fil^{-i}(S_n[z_n^{-1}]) \subseteq \Fil^r \scrM \otimes_{S} \Fil^{-r}(S_n[z_n^{-1}]).$$
	Indeed, recalling that $z_n = E\varphi^{-1}(E)\cdots \varphi^{1-n}(E)$,
	we see that $(z_n/E)\in \s_n$ and compute that any simple tensor on the left side has the form
	$$m \otimes sz_n^{-i} = m \otimes sz_n^{r-i} z_n^{-r} = E^{r-i}m \otimes s (z_n/E)^{r-i} z_n^{-r}$$
	with $m\in \Fil^i\scrM$, and this
	lies in $\Fil^r \scrM \otimes \Fil^{-r}(S_n[z_n^{-1}])$
	thanks to the very definition (\ref{FilDefn}) of $\Fil^i$.
	
	On the other hand, it follows immediately from
	Lemma \ref{specialbasis} that any
	$\xi_n\in \Fil^r \scrM \otimes_{S} \Fil^{-r}(S_n[z_n^{-1}])$
	may be written in the form
	$$
	\xi_n=\left((\alpha_1,\ldots,\alpha_d)x_n+(e_1,\ldots,e_d)y_n\right) \otimes z_n^{-r}
	= z_n^{-r}(e_1,\ldots,e_d)\cdot \left(Ax_n+y_n\right)
	$$
	for vectors $x_n\in S_n^d$ and $y_n\in (\Fil^p S_n)^d$.
\end{proof}

\begin{lemma}\label{KeyB}
	Assume $p > 2$ and
	let $d$ and $r$ be positive integers with $r < p-1$. Let $A$  be a $d\times d$
	matrix with entries in $\s=\s_0$ such that there exists a $d\times d$
	matrix $B$ with entries in $\s$ satisfying $BA = E^r I_d$.
	Let $x_1$ a vector in $S_1^d$, and assume
	that for all $n\ge 2$ there is a vector $x_n\in S_n^d$ with
	\begin{equation}
		\varphi(x_{n})= Ax_{n-1}.
	\end{equation}
	Then all coordinates of $x_1$ lie in $\s_1$.
\end{lemma}

\begin{proof}
	For ease of notation, if $j$ is a positive integer, we will write $\Fil^j S_n^d$ for the submodule of $S_n^d$ consisting of vectors
	all of whose components lie in $\Fil^j S_n$.  Suppose given a sequence $\{x_n\}_{n\ge 1}$ as above.
	We will prove that for any $n > 1$, if $x_n$ can be written as a sum $x_n=y_n+y'_n$ with $y_n\in \s_n^d$ and $y_n'\in \Fil^j S_n^d$,
	then $x_{n-1}$ can be written $x_{n-1}=y_{n-1}+y_{n-1}'$ where $y_{n-1}\in \s_{n-1}^d$ and
	$y_{n-1}'\in \Fil^{pb(j)-r} S_{n-1}^d$.  So assume $x_n=y_n+y_n'$ with $y_n$ and $y_n'$ as above.  Applying Frobenius coordinate-wise
	and using Lemma \ref{KeyA} and our hypotheses, we find that
	\begin{equation*}
		Ax_{n-1}=\varphi(x_n)=v_{n-1}+ v_{n-1}'
	\end{equation*}
	with $v_{n-1}\in \s_{n-1}^d$ and $v_{n-1}'\in \Fil^{pb(j)} S_{n-1}^d$.  Multiplying both sides by $B$ then gives
	\begin{equation*}
		E^r x_{n-1}=B\varphi(x_n)=Bv_{n-1}+ Bv_{n-1}' = w_{n-1} + w_{n-1}'
	\end{equation*}
	with $w_{n-1}\in \s_{n-1}^d$ and $w_{n-1}'\in \Fil^{pb(j)} S_{n-1}^d$.
	Now $r < p-1 \le pb(j)$, from which it follows
	that $w_{n-1}=E^r y_{n-1}$ for some $y_{n-1}\in \s_{n-1}^d$ thanks to Lemma \ref{Lem:intersect}.
	We may then write $w_{n-1}' = E^r y_{n-1}'$ with $y_{n-1}'\in \Fil^{pb(j)-r} S_{n-1}^d[1/p]$.
	But since $x_{n-1}=y_{n-1}+y_{n-1}'$ with $x_{n-1}$ and $y_{n-1}$ both having coordinates in $S_{n-1}$,
	we conclude again using Lemma \ref{Lem:intersect}
	that $y_{n-1}'$ has coordinates in $S_{n-1}\cap \Fil^{pb(j)-r} S_{n-1}[1/p]=\Fil^{pb(j)-r} S_{n-1}$
	as desired.
	
	To complete the proof, we observe that since $S_n = \s_n + \Fil^p S_n$, it follows from repeated applications
	of the above fact that $x_1 = y_1 + y_1'$ with $y_1\in \s_1^d$ and $y_1'$ in $\Fil^{j_n} S_1^d$,
	with $j_n$ determined recursively by $j_1=p$ and $j_n = pb(j_{n-1})-r$ for $n> 1$.
	From the definition of $b(\cdot)$ in Lemma \ref{KeyA} and our hypothesis $r < p-1$, we compute that for $n\ge 1$
	$$j_{n+1}-j_n = pb(j_n) - r -j_n \ge
	 (p-3)j_n+ (p-2)\left(\frac{j_n}{p-1}-1\right).$$
	Using the hypothesis $p>2$ and induction on $n$ with base case $j_1=p$, we deduce
	that $j_{n+1} > j_n$ for all $n > 0$, so that $\{j_n\}_{n\ge 0}$
	is an {\em increasing} sequence of positive integers.
	Taking $n\rightarrow \infty$ then gives $x_1\in \s_1^d$ as desired.
\end{proof}

\begin{lemma}\label{KeyC}
	In the situation of Lemma $\ref{KeyB}$,
	let $x_1\in S_1^d$ and $y_1\in \Fil^p S_1^d$
	and suppose that for all $n\ge 1$ there are vectors
	$x_{n+1} \in S_{n+1}^d$ and $y_{n+1}\in \Fil^p S_{n+1}^d$
	with
	\begin{equation}
		\varphi(Ax_{n+1} + y_{n+1})= \varphi(A)(Ax_{n}+y_{n}).\label{AxyRecursion}
	\end{equation}
	Then there exists a vector $w\in \s_1^d$ such that
	$$Ax_n+y_n = A\varphi^{-1}(A)\ldots \varphi^{1-n}(A)\cdot\varphi^{1-n}(w).$$
	for $n\ge 1$.  In particular, $Ax_n+y_n$ has all coordinates in $\s_n$.
\end{lemma}

\begin{proof}
	Let $n\ge 1$.
	Since $S_n = \s_n + \Fil^p S_n$, we may and do assume that $x_n$
	has all coordinates in $\s_n$.
	Let us write $T_n$ for the closure of the subring $\s_n[E^p/p]$ inside $S_n$.
	We first claim that $y_n$ has all coordinates in $T_n$.  To see this, observe that
	as $y_{n+1}\in \Fil^p S_{n+1}^d$ by hypothesis, we may write
	$y_{n+1} = \sum_{j\ge p} w_j E^j/j!$ with $w_j$ a vector in $\s_{n+1}^d$ for all $j$.
	Using the recursion (\ref{AxyRecursion}) to isolate $y_n$, we find
	\begin{equation*}
		\varphi(A)y_n = \varphi(A)(\varphi(x_{n+1})-Ax_n)+ \sum_{j\ge p} \varphi(w_j) c_0^j \frac{p^j}{j!}.
	\end{equation*}
	Multiplying both sides by $\varphi(B)$ and dividing by $p^r$ we find
	$$c_0^r y_n = c_0^r(\varphi(x_{n+1})-Ax_n)+ \varphi(B)\sum_{j\ge p} \varphi(w_j) c_0^j \frac{p^{j-r}}{j!},$$
	and a standard calculation shows that for $j \ge p$ and $r \le p-1$ we have $v_p(p^{j-r}/j!)\ge 0$.
	As the right side then clearly has coordinates in $T_n$, our claim follows.
	
	Now we may write $y_n = \sum_{i\ge 0} w_i (E^p/p)^i$ with $w_i\in \s_n^d$ for all $i$.
	Since $y_n$ has coordinates in $\Fil^p S_n$, we must have $w_0 = E^p v_0$ for some $v_0\in \s_n^d$
	and we compute that
	$$py_n = E^p pv_0 + E^p \sum_{i\ge 1} w_i \left(\frac{E^p}{p}\right)^{i-1}.$$
	In particular, $py_n = E^p y_n' = A(BE^{p-r}y_n')=At_n$ for some $t_n$ with coordinates in $T_n$.
	Then $p(Ax_n + y_n) = A(px_n + t_n) = A s_n$ with $s_n$ a vector with all coordinates in $T_n\subseteq S_n$.
	Multiplying (\ref{AxyRecursion}) by $p$ and replacing $p(A{x_n}+y_n)$ by $As_{n}$ gives the recurrence
	$$ \varphi(s_n) = A s_{n-1},$$
	for all $n>1$, which forces $s_1\in \s_1^d$ thanks to Lemma \ref{KeyB}.
	For $n\ge 1$ we then have
	\begin{equation}
		p(Ax_n + y_n) = As_n = A\varphi^{-1}(A)\cdots \varphi^{1-n}(A)\varphi^{1-n}(s_1).
		\label{DivByp}
	\end{equation}
	To complete the proof, it therefore suffices to prove that $s_1$ has all coordinates
	divisible by $p$ in $\s_1$.

	Multiplying \ref{DivByp} through by $C:=\varphi^{1-n}(B)\cdots \varphi^{-1}(B)B$
	gives
	\begin{equation}
		pC(Ax_n+y_n) = E^r \varphi^{-1}(E^r)\cdots \varphi^{1-n}(E^r) \varphi^{1-n}(s_1) = z_{n}^r\varphi^{1-n}(s_1).
	\end{equation}
	Since $pC(Ax_n+y_n)$ has coordinates  in $pS_n$, we certainly have that all coordinates of
	$z_{n}^r \varphi^{1-n}(s_1)$ are zero in $S_n/pS_n$.
	On the other hand, from the very definition of $S_n$, we have an injection
	$k[u_n]/(u_n^{ep^{n+1}})\hookrightarrow S_n/pS_n$, where $e$ is the $u_0$-degree of $E=E(u_0)$.
	Write $s_1=(s_{11}(u_1),\ldots,s_{1d}(u_1))$ with $s_{1j}\in \s_1$,
	so that $\varphi^{1-n}(s_1)$ has coordinates
	$s_{1j}^{\sigma^{1-n}}(u_n)\in \s_n$ for $1\le j\le d$. Since
	$$z_{n} \equiv u_0^eu_1^e\cdots u_{n-1}^e \equiv u_n^{pe\frac{p^{n}-1}{p-1}} \bmod p,$$
	it follows from the above that
	the reduction modulo $p$ of each coordinate $s_{1j}^{\sigma^{1-n}}(u_n)\in \s_n$ is divisible by
	$u_n^{e i_n}$ in $k[\![u_n]\!]$ for {\em all $n\ge 1$}, where
	$$i_n = p\left(p^{n} - r\frac{p^{n}-1}{p-1}\right).$$	
	This implies that $s_{1j}(u_1)\bmod p$ is divisible by $u_1^{ei_n}$ for {\em all $n\ge 1$} and $j$.
	Again a straightforward calculation using the hypothesis $r < p-1$
	shows that $i_n\rightarrow \infty$ as $n\rightarrow\infty$, and we conclude that $s_{1j}(u_1)\equiv 0\bmod p$
	for all $j$, whence $s_1$ has all coordinates divisible by $p$ in $\s_1$, as desired.
\end{proof}

Let $(M,\Fil M,\varphi_1)$ be an arbitrary object of $\Mod_{\s}^{\varphi,r}$
and let $\varphi_M:M\rightarrow M$ be as in (\ref{phiDefonM}).
Give the ring $\s_n[z_n^{-1}]$ the $\Z$-filtration by powers of $z_n$,
and for ease of notation, set
$$M_n:=\Fil^0(M \otimes_{\s} \s_{n}[z_{n}^{-1}]).$$
\begin{lemma}\label{Lifting}
	For $n\ge 0$ and $x\in M_n$, there exists $y\in M_{n+1}$
	with $(\varphi_M\otimes \varphi)(y)=x$.  Moreover, $y$ is unique.
\end{lemma}

\begin{proof}
	Since the image of $\varphi_{M,r}: \Fil^r M\rightarrow M$
	generates $M$ as an $\s$-module, every element of $M_n$
	is a sum of elements of the form $\varphi_{M,r}(m) \otimes (s/z_n^r)$, for appropriate $m\in \Fil^r M$
	and $s\in \s_n$.  Consider the element $m \otimes (\varphi^{-1}(s)/z_{n+1}^r)$,
	which lies in $M_{n+1}=\Fil^0(M \otimes_{\s} \s_{n+1}[z_{n+1}^{-1}])$.  Then:
	\begin{align*}
		(\varphi_M\otimes \varphi)(m \otimes (\varphi^{-1}(s)/z_{n+1}^r)) &=
		(\varphi_M(m))\otimes (s/\varphi(z_{n+1}^r)) \\
		&=\varphi(E)^r \varphi_{M,r}(m) \otimes (s/\varphi(E)^rz_n^r)\\
		& = \varphi_{M,r}(m) \otimes (s/z_n^r)
	\end{align*}
	This proves the existence of $y$ as in the statement of the lemma.
	Uniqueness follows immediately from the fact that $\varphi_M\otimes \varphi$,
	viewed as a self-map of $M\otimes_{\s} \s_{n+1}[z_{n+1}^{-1}]$, is injective.
\end{proof}

\begin{remark}\label{FilRem}
	The Lemma shows the stronger fact that any $x\in M\otimes_{\s} \Fil^{-r}(\s_n[z_n^{-1}])$
	has a unique preimage under $\varphi_M\otimes \varphi$ in $\Fil^r M \otimes \Fil^{-r}(\s_n[z_{n+1}^{-1}])$.
\end{remark}

Now let $(\scrM,\Fil^r\scrM,\varphi_{\scrM,r}):=\u{\scrM}(M)$ be the functorially associated
object of $\Mod_{S}^{\varphi,r}$, so $\scrM = M\otimes_{\s} S$
and $\Fil^r \scrM$ is
$S$-submodule of $\scrM$
generated by the images of $M\otimes_{\s} \Fil^r S$ and $\Fil^r M \otimes_{\s} S$
under the obvious maps.  As such, we have a canonical inclusion of $\s_n$-modules:
\begin{equation*}
	\xymatrix{
		{\iota_n:M_n:=\Fil^0(M \otimes_{\s} \s_{n}[z_{n}^{-1}])}\ar[r] &{ \Fil^0(\scrM\otimes_{S} S_n[z_n^{-1}])}
		}
\end{equation*}
that is $\varphi$-compatible.
We also have an obvious {\em isomorphism} $\tau:M\rightarrow M_0$ given by $m\mapsto m\otimes 1$.
Given $m\in M$, for $n\ge 0$ we then define $\xi_n\in M_n$ to be the unique element
of $M_n$ satisfying
$$(\varphi_M\otimes\varphi)^{(n)}(\xi_n)=\tau(m);$$
this exists thanks to Lemma \ref{Lifting}.
We obtain a map:
\begin{equation}
	M \rightarrow \u{M}(\scrM) = \varprojlim_{\varphi,n} \Fil^0(\scrM\otimes_S S_n[z_n^{-1}])
	\qquad\text{given by}\qquad
	m \mapsto \{\iota_n(\xi_n)\}_{n\ge 0}.\label{MTMap}
\end{equation}

\begin{lemma}\label{Isom}
	The map $(\ref{MTMap})$ is a natural isomorphism of filtered $\varphi$-modules
	over $\s$.
\end{lemma}

\begin{proof}
	We first prove that (\ref{MTMap}) is an isomorphism at the level of $\s$-modules.
	Suppose that $\{\xi_n\}_{n\ge 0}$ is an arbitrary element of
	$\varprojlim_{\varphi,n}(\scrM\otimes_S S_n[z_n^{-1}])$.
	It suffices to prove that $\xi_0$ lies in the image of the
	canonical inclusion
	$$
	\xymatrix{
		{\iota_0\circ\tau: M} \ar[r]^-{\simeq}_-{\tau} & {M_0=\Fil^0(M\otimes_{\s} \s_0)}
		\ar[r]_-{\iota_0} & {\Fil^0(\scrM\otimes_S S_0)}
		}.
	$$
	Indeed, then projection $\{\xi_n\}_{n\ge 0}\mapsto \xi_0$ followed by
	the inverse of $\iota_0\circ \tau$ on its image provides the desired
	inverse map to (\ref{MTMap}).
		
	To do this, we identify $M$ with its image under $\iota_0\circ\tau$
	and
	compute with bases.
	The map $\varphi_{M,r}:\varphi^*\Fil^r M\rightarrow M$ is a linear isomorphism
	of $\s$-modules, so since $\varphi:\s\rightarrow \s$ is faithfully flat,
	$\Fil^r M$ is finite and free over $\s$ with rank equal to that of $M$;
	this fact also follows easily from the discussion of Remark \ref{ClassicalKisin}.
	Fix an $\s$-basis $\alpha_1,\ldots,\alpha_d$ for $\Fil^r M$
	and set $e_i:=\varphi_{M,r}(\alpha_i)$, so that $e_i$ is then an $\s$-basis of
	$M$.  Since $E^rM\subseteq \Fil^r M$, we obtain matrices
	$A, B\in M_d(\s)$ determined by the conditions
	$$(\alpha_1,\ldots,\alpha_d)=(e_1,\ldots e_d)A\quad\text{and}
	\quad (e_1,\ldots,e_d)E^r = (\alpha_1,\ldots,\alpha_d)B$$
	so that $AB=BA=E^r$.  Note that the associated Breuil module $\scrM$
	admits the ``explicit" description as in Lemma \ref{specialbasis}.
	
	Thanks to Lemma \ref{alpharep}, for all $n\ge 1$ we may write
	$$\xi_{n} = z_{n}^{-r}(e_1,\ldots,e_d)\cdot (Ax_{n}+y_{n})$$
	for vectors $x_{n}\in S_{n}^d$ and $y_{n}\in \Fil^p S_{n}^d$.
	For $n\ge 1$ we then compute
	\begin{align*}
		\xi_{n} = (\varphi_{\scrM}\otimes\varphi)(\xi_{n+1})&=
		\varphi(E)^{-r}z_{n}^{-r}(e_1,\ldots,e_d)\varphi(B)\varphi(Ax_{n+1}+y_{n+1}).
	\end{align*}
	Multiplying both sides by $z_{n}^r\varphi(E)^r$, using the definition of $\xi_{n}$,
	and comparing coefficients of $e_i$ gives
	\begin{align*}
		\varphi(E)^r(Ax_{n}+y_{n}) = \varphi(B)\varphi(Ax_{n+1}+y_{n+1})
	\end{align*}
	as (column) vectors in $S_{n}^d$.
	Multiplying this equality through by $\varphi(A)$,
	and cancelling the resulting factor of $\varphi(E)^r=\varphi(A)\varphi(B)$
	from both sides finally yields the recurrence
	\begin{align*}
		\varphi(Ax_{n+1}+y_{n+1}) = \varphi(A)(Ax_{n}+y_{n}).
	\end{align*}
	for $n\ge 1$.
	But now we are in precisely the situation of Lemma \ref{KeyC}, which
	guarantees that $Ax_1+y_1=Aw_1$ for some $w_1\in \s_1 ^d$ so that
	$$
	\xi_0 = (\varphi_{\scrM}\otimes \varphi)(\xi_1)
	= \varphi(z_1)^{-r}(e_1,\ldots,e_d)\varphi(B)\varphi(A)\varphi(w_1)
	=(e_1,\ldots,e_d)\varphi(w_1)
	$$
	lies in $M$, as desired.
	
	That the map $(\ref{MTMap})$ is Frobenius-compatible and carries $\Fil^r M$ into $\Fil^r \u{M}(\scrM)$
	is clear from definitions.
	To check that it induces an isomorphism on filtrations, it suffices to prove that projection
	$\{\xi_n\}_{n\ge 0}\mapsto \xi_0$ is filtration-compatible.
	This amounts to the assertion that $\Fil^r \scrM \cap M \subseteq \Fil^r M$ inside $\scrM$.
	To verify this, as before, we may write any element of $\Fil^r \scrM$ as
	$(e_1,\ldots, e_d)(Ax+y)$ with $x\in \s^d$ and $y\in \Fil^p S^d$.  If this
	is equal to some element $(e_1,\ldots, e_d)w$ of $M$ with $w\in \s^d$, then
	we must have $Ax+y = w$ in $S^d$.  Multiplying both sides by $B$ gives
	$E^r x + By = Bw$ so since $By\in \Fil^p S^d$ we deduce that the coordinates
	of $Bw$ lie in $\Fil^r S \cap \s = E^r \s$ thanks to Lemma \ref{Lem:intersect}.
	Then since $x\in \s^d$, it follows that $By$ has coordinates in $\Fil^p S \cap \s=E^p\s$,
	and we may write $Bw= E^r v = BA v$ for some $v\in \s^d$.  This implies that $w=Av$
	and hence that $(e_1,\ldots, e_d)w=(\alpha_1,\ldots,\alpha_d)v$ lies in $\Fil^r M$
	as desired.	
\end{proof}

\begin{corollary}\label{IsAFun}
	Let $\scrM$ be any object of $\Mod_{S}^{\varphi,r}$.
	Then $\u{M}(\scrM)$ is an object of $\Mod_{\s}^{\varphi,r}$.
\end{corollary}

\begin{proof}
	This follows immediately from Lemmas \ref{specialbasis} and \ref{Isom}.
\end{proof}

We now have functors $\u{\scrM}:\Mod_{\s}^{\varphi,r}\rightarrow \Mod_{S}^{\varphi,r}$
and $\u{M}:\Mod_S^{\varphi,r}\rightarrow \Mod_{\s}^{\varphi,r}$ and a functorial
isomorphism $\u{M}\circ\u{\scrM}\simeq \id$ on $\Mod_{\s}^{\varphi,r}$.
To complete the proof of Theorem \ref{MainThm},
it therefore remains to exhibit a natural transformation $\u{\scrM}\circ \u{M}\simeq \id$
of functors on $\Mod_S^{\varphi,r}$.

Let $(\scrM,\Fil^r\scrM,\varphi_{\scrM,r})$ be any object of $\Mod_S^{\varphi,r}$.
We define an $S$-linear map
\begin{equation}
	\u{\scrM}(\u{M}(\scrM))=S\otimes_{\s}\varprojlim_{\varphi,n} \Fil^0(\scrM\otimes_S S_n[z_n^{-1}]) \rightarrow \scrM\otimes_S S_0 \simeq \scrM
	\quad\text{by}\quad
	s\otimes \{\xi_n\}\mapsto s\cdot \xi_0.\label{Smap}
\end{equation}

\begin{lemma}
	The map $(\ref{Smap})$ is a natural isomorphism of filtered $\varphi$-modules over $S$.
\end{lemma}

\begin{proof}
	Naturality in $\scrM$ is clear, as is compatibility with Frobenius.
	By the very definition (\ref{BC}) of $\u{\scrM}$, the submodule $\Fil^r(\u{\scrM}(\u{M}(\scrM)))$ is generated
	by the images in $S\otimes_{\s} \u{M}(\scrM)$ of $S\otimes_{\s} \Fil^r \u{M}(\scrM)$ and
	$\Fil^r S\otimes_{\s} \u{M}(\scrM)$, so due to Definition \ref{Def:Mfunctor},
	any element of this submodule is a sum of
	simple tensors $s\otimes \{\xi_n\}$ with either $s\in \Fil^r S$ or $\xi_0\in \Fil^r\scrM$.
	Since $(\Fil^r S)\scrM\subseteq \Fil^r\scrM$, it follows at once that the map
	(\ref{Smap}) is compatible with filtrations.
	
	Let us prove that (\ref{Smap}) is an isomorphism.  Thanks to Corollary \ref{IsAFun},
	the map (\ref{Smap}) is an $S$-linear map of {\em free} $S$-modules of the same rank,
	so it suffices to prove that it is surjective.  Let $(e_1,\ldots,e_d)$, $(\alpha_1,\ldots,\alpha_d)$,
	$A$ and $B$ be as in Lemma \ref{specialbasis}.  It is clearly enough to prove that
	$e_i$ is in the image of (\ref{Smap}) for each $i$.  For $n\ge 1$ we define
	$$(\xi_{1,n},\ldots,\xi_{d,n}):=z_n^{-r}((e_1\ldots,e_d)A)\varphi^{-1}(A)\cdots \varphi^{1-n}(A).$$
	As $(\alpha_1, \ldots, \alpha_d )=(e_1,\ldots,e_d)A$ and $\alpha_i$ lies in $\Fil^r\scrM$, this really is an element of $\Fil^0(\scrM\otimes_S S_n[z_n^{-1}])$
	for $n\ge 1$ and we set $\xi_{i,0}:=e_i$.
	We then compute for $n\ge 1$
	\begin{align*}
		(\varphi_{\scrM}\otimes \varphi)(\xi_{1,n},\ldots,\xi_{d,n}) &= \varphi(E)^{-r}z_{n-1}^{-r} (e_1,\ldots,e_d)
		\varphi(B)\varphi(A)A\cdots\varphi^{2-n}(A)\\
		&=z_{n-1}^{-r} (e_1, \ldots, e_d )A\varphi^{-1}(A)\cdots \varphi^{2-n}(A) =(\xi_{1,n-1},\ldots,\xi_{d,n-1})
	\end{align*}
	so that $\xi_i:=\{\xi_{i,n}\}_{n\ge 0}$ lies in $\u{M}(\scrM)$ and $1\otimes \xi_i$ maps to $e_i$ via
	(\ref{Smap}).
	
	Finally, we must check that the map on $\Fil^r$'s is an isomorphism, and to do so
	it suffices to prove that it is surjective.  We know from Lemma \ref{specialbasis}
	that any element $m\in \Fil^r\scrM$ may be expressed as $m=(e_1,\ldots e_d)(Ax+y)$
	where $x\in \s^d$ and $y\in \Fil^r S^d$.
	For $n\ge 1$ define
	$$\nu_n:=z_n^{-r}(e_1,\ldots,e_d)A\varphi^{-1}(A)\cdots \varphi^{1-n}(A)\varphi^{-n}(A)\varphi^{-n}(x),$$
	which again lies in $\Fil^0(\scrM \otimes_{S} S_n[z_n^{-1}])$, and put $\nu_0:=(e_1,\ldots,e_d)Ax$,
	which lies in $\Fil^r\scrM$.  Then as before one checks that $\nu:=\{\nu_n\}$
	is an element of $\Fil^r \u{M}(\scrM)$ with $1\otimes\nu$ mapping to $\nu_0$.
	Since $y=(y_1,\ldots,y_d)^T\in \Fil^r S^d$, the element $\eta:=\sum_i y_i \otimes \xi_i$
	lies in $\Fil^r S\otimes \u{M}(\scrM)$ and maps to $(e_1,\ldots,e_d)y$ under (\ref{Smap}).
	Thus, the sum $1\otimes \nu + \eta$ is an element of $\Fil^r \scrM(\u{M}(\scrM))$
	mapping to $m$, and the map on filtrations is surjective, as desired.
\end{proof}

\section{Lattices in Galois representations}

In this section, we briefly review the relationship between
the semilinear algebra categories of \S\ref{BTconst} and
(stable lattices in) Galois representations.

We keep the notation of \S\ref{Intro}, and begin by recalling the definitions of the period rings that we will need.
Let
$R:=\varprojlim \O_{\overline K}/ p \O_{\overline K}$, with the projective limit taken along the map $x\mapsto x^p$.
Then $R$ is a perfect valuation ring of equicharacteristic $p$ and residue field $\o{k}$, equipped
with a natural coordinate-wise action of $G_K$.  We put $A_{\inf}:=W(R)$, and denote by
$\theta_{\inf} : A_{\inf} \to \O_{\C_K}$ the unique ring homomorphism
lifting the  projection $R \to \O_{\overline K}/ p$
onto the first factor in the inverse limit. We denote by $A_{\cris}$ the $p$-adic completion
of the divided power envelope of $A_{\inf}$ with respect to the ideal $\ker(\theta_{\inf})$.
As usual, we write $B_{\cris}^+= A_{\cris}[1/p]$ and  define $B_{\dR}^+$ to be the $\ker(\theta_{\inf}[1/p])$-adic completion of $A_{\inf}[1/p]$. For any subring $A \subset B^+_{\dR}$, we define
$\Fil^i A = A \cap (\ker(\theta_{\dR}))^iB^+_{\dR}$, with $\theta_{\dR}:B_{\dR}^+\twoheadrightarrow \C_K$ the map induced by $\theta_{\inf}$.

Recall that we have fixed a compatible sequence $\{\pi_i\}_{i\ge 0}$ of $p$-power roots
of our fixed uniformizer $\pi_0\in K$.
Then $\{\pi_i\}_{i \geq n}$ defines an element $\u \pi_n \in R$,
and we write $[\underline \pi_n  ]\in A_{\inf}$ for its Techm\"uller lift.
For each $n$, we then embed the $W(k)$-algebra $W(k)[u_n]$ into $A_{\inf}\subset A_{\cris}$ by the map
$u_n\mapsto [\underline \pi_n ]$.   These maps
extend to embeddings $\s_n \hookrightarrow A_{\inf}$ which intertwine
the given Frobenius endomorphism on $\s_n$ with the Witt vector Frobenius on $A_{\inf}$,
and which are compatible with the $W(k)$-algebra inclusions $\s_n\hookrightarrow \s_{n+1}$
that identify $\varphi(u_{n+1})=u_n$.
As before, we omit the subscript when it is zero, and simply write $\s = \s_0$ and $u = u_0$.

Recall that $S_n $ is   the $p$-adic completion of the divided power
envelope of $\s_n$ with respect to the ideal generated by $E(u_0)=E_n (u_n )$,
equipped with the $p$-adic topology.
We write $\Fil ^m  S\subset S $ for the closure of the ideal generated by
$\gamma_i (E(u)):= \frac{E_n(u_n)^i}{i!}$ with $ i \geq m$.
Using the fact that $\ker(\theta_{\inf})$ is principally generated by $E([\u{\pi}_0])=E_n([\u\pi_n ])$,
it is not difficult to prove that
the embeddings $\s_n\hookrightarrow A_{\inf}$ uniquely extend to continuous $W(k)$-algebra
embeddings $ S_n \inj A_{\cris}$ compatible with Frobenius $\varphi$ and filtration.
As in \S\ref{Intro}, we write $K_n:=K(\pi_n)$ and set $K_\infty := \cup_n K_n$.
We define $G_\infty:= \Gal (\overline K / K_\infty)$ and note that we in fact have
$\s_n \subset A_{\inf}^{G_\infty}$ and $S_n \subset A_{\cris}^{G_\infty}$.

With these preliminaries, we now define certain functors from the categories of
(filtered) Breuil--Kisin and Breuil modules to the category of Galois representations on
finite free $\Z_p$-modules.

Let $M \in \Mod_{\s}^{\varphi,r}$ be a  filtered Breuil--Kisin module.
Remembering that $\Fil ^ r A_{\inf} = E(u )^r A_{\inf}$, we define
$\varphi_{A_{\inf}, r} : \Fil ^r A_{\inf} \to A_{\inf}$ by $\varphi_{A_{\inf}, r} (E(u)^r x):= \varphi (x),$
and set
\begin{equation}\label{eqn-deineTS}
\u T_{\s} (M):  = \Hom_{\s, \Fil^r, \varphi_r} (M, A_{\inf}),
\end{equation}
which, as one checks easily, is naturally a $\Z_p[G_\infty]$-module.
Similarly, the restriction of $\varphi$ on $A_{\cris}$ to $\Fil^r A_{\cris}$
has image in $p^r A_{\cris}$, so we may define
$\varphi_{A_{\cris, r}} : \Fil ^r A_{\cris}  \to A_{\cris}$ by
$\varphi_{A_{\cris}, r } = \varphi _{A_{\cris}}/p^r$.
For any quasi-Breuil module  $\scrM \in \Mod_S^{\varphi, r}$ we may then attach the $\Z_p[G_{\infty}]$-module
\begin{equation} \label{eqn-defineTcris}
\u T_{\cris} (\scrM) : = \Hom_{S, \Fil^r, \varphi_r } (\scrM, A_{\cris}).
\end{equation}

Before proceeding further, we recall a variant of the functor $\u{T}_{\s}$ on the
category of (classical) Breuil--Kisin modules of Remark \ref{ClassicalKisin}.
Let $M\in \Mod_{\s}^{\varphi,r}$, and let $(\m,\varphi_{\m})$ be the associated
classical Breuil--Kisin module.  Then, as in (\ref{BK-dictionary}) and (\ref{BKFil-Alt}), we have
$M = \varphi ^* \m$ with $\varphi_M=\varphi\otimes\varphi_{\m}$ and
$$\Fil ^r M = \Fil ^r \varphi ^*\m = \{ x \in \varphi ^*\m | (1 \otimes \varphi )(x) \in E^r\m\}.$$
It is clear from these descriptions that the restriction of $\varphi_M$ to $\Fil^r M$
has image contained in $\varphi(E)^r M$, so we may and do define $\varphi_{M,r}:=\varphi(E)^{-r}\varphi_M$
on $\Fil^r M$.
We then set
$$ T_\s (\m) : = \Hom_{\s, \varphi} (\m , A_{\inf}).$$

\begin{lemma}\label{lem-dual}
With notation as above
\begin{enumerate}
\item There is a natural isomorphism $\u T_\s (M) \simeq T _\s (\m)$ of $\Z_p [G_\infty]$-modules;\label{KisinGalois}
\item There is an isomorphism of functors $ \u T_\s \simeq \u T_{\cris}\circ \u \scrM$ on $\Mod_{\s}^{\varphi,r}$.
\end{enumerate}

\end{lemma}
\begin{proof} By \cite{LiuT-CofBreuil}[Lemma 3.3.4], there is a natural isomorphism  $ T_\s(\m) \simeq \u T_{\cris} (\u \scrM(\m))$ of $\Z_p[G_\infty]$-modules,  so it suffices to prove (\ref{KisinGalois}). Using the
relation $\varphi_{\star,r}(E^rx) = \varphi_{\star}(x)$ for $\star=M, A_{\inf}$, one shows that
there is a canonical map
\begin{equation}
	\xymatrix{
		{\u T_\s (M) =  \Hom_{\s, \Fil^r, \varphi_r} (M, A_{\inf})}\ar[r] & {\Hom_{\s, \varphi} (M , A_{\inf})}
		}\label{GalModMap}
\end{equation}
of $\Z_p[G_{\infty}]$-modules that is visibly injective.  We claim it is surjective as well,
and hence an isomorphism.
To see this, let $f \in \Hom_{\s, \varphi} (M , A_{\inf})$ and $x \in \Fil^r M$ be arbitrary.
Again using the above relation, we compute
\begin{align*}
	\varphi_{A_{\inf}}(f(x)) = f(\varphi_M(x)) = f(\varphi_{M,r}(E^rx)) = f(\varphi(E^r)\varphi_{M,r}(x)) =
	\varphi(E)^r f(\varphi_{M,r}(x)),
\end{align*}
so recalling that $\varphi_{A_{\inf}}$ is an {\em automorphism} of $A_{\inf}$ we conclude
that $f(x) = E^r \varphi_{A_{\inf}}^{-1} f(\varphi_{M,r}(x))$ and $f$ carries $\Fil^r M$
into $\Fil^r A_{\inf}$.  Written another way, this last equality reads
$$f(\varphi_{M,r}(x)) = \varphi_{A_{\inf}}(E^{-r} f(x)) = \varphi_{A_{\inf},r}(f(x))$$
and $f$ is compatible with $\varphi_r$'s.  This shows that (\ref{GalModMap}) is indeed
an isomorphism as claimed.

To complete the proof, it now suffices to exhibit a natural isomorphism of $\Z_p[G_{\infty}$]-modules
\begin{equation}
\xymatrix{
	{T_\s(\m) := \Hom_{\s , \varphi} (\m,A_{\inf})} \ar[r]^-{\simeq} &
	{\Hom_{\s,\varphi}(\varphi^*\m,A_{\inf})=\Hom_{\s , \varphi} (M , A_{\inf})}
	}.\label{Map:GaloisCompare}
\end{equation}
To do this, for $f \in T_\s(\m)$ we define $\iota (f) \in \Hom_{\s } (\varphi ^*\m, A_{\inf})$
by
$$  \iota(f)(\sum_i a_i \otimes m _i):= \sum _i a_i \varphi (f(m_i))$$
Since $f$ is compatible with Frobenius, it is straightforward to see that the
same is true of $\iota(f)$, so $\iota$ induces a map (\ref{Map:GaloisCompare}).
Using the fact that $\varphi_{A_{\inf}}$ is bijective, one then checks
easily that this map is an isomorphism as claimed.
\end{proof}

In order to use the category of Breuil modules to study $G_K$ representations
(rather than just $G_{\infty}$-representations), we require the additional structure
of a monodromy operator.
Let $V$ be a crystalline representation with Hodge-Tate weights in $\{0, \dots, r\}$ and $T \subset V$ a
$G_K$-stable $\Z_p$-lattice. We denote $T ^\vee= \Hom_{\Z_p} (T , \Z_p)$ the $\Z_p$-linear {dual} of
$T$ and put $V^\vee := T^\vee [1/ p]$. For ease of notation, we write $D:= D_{\cris}(V^\vee)$
for the associated filtered $(\varphi,N)$-module; of course $N_D=0$ as $V$ is crystalline.

By \cite{Breuil-Griffith}, we can functorially promote $D$ to
a \emph{filtered $(\varphi, N)$-module}  $\scrD(V)= (\scrD, \{\Fil^j \scrD\}_{j}, \varphi_\scrD, N_\scrD)$ over $S[1/p]$
by defining

\begin{itemize}
\item $\scrD  := S \otimes_{W(k)} D$ with $\varphi_\scrD:=\varphi_S\otimes \varphi_D$
\item $N_{\scrD}:=N_S\otimes \id + \id\otimes N_D=N_S\otimes \id$,
where $N_S:S\rightarrow S$ is the unique continuous $W$-linear derivation with $N(u)=-u$.
\item $\Fil ^j  \scrD $ is defined inductively by setting $\Fil ^0  \scrD :=  \scrD$ and
\begin{equation}\label{eqn-fil}
\Fil ^{j}  \scrD  = \{ x \in  \scrD  | N(x) \in \Fil ^{j-1}  \scrD ,\ \text{and}\  f_{\pi_0} (x) \in \Fil^j D_K  \}
\end{equation}
where $f_{\pi_0} :  \scrD \to D_K := \O_K \otimes_{W(k)} D $ is the projection induced by
the map $f_{\pi_0} : S\twoheadrightarrow \O_K$ sending $u$ to $\pi_0$.
\end{itemize}
The reader can consult the precise definition of filtered $(\varphi, N)$-modules over $S[1/p]$ in \cite{LiuT-CofBreuil},
which we do not need here.   Following \cite{Breuil-SDLattice}, we introduce:
\begin{definition}\label{Def-SDlattice}
\emph{A strongly divisible $S$-lattice $\scrM$ inside $\scrD= \scrD(V)$} is a finite free $S$-submodule $\scrM \subset \scrD$ that is stable under $\varphi_{\scrD}$ and satisfies
\begin{itemize}
 \item $\scrM [1/p]= \scrD$;
 \item $\varphi_{\scrD} (\Fil ^r \scrM) \subset p ^r \scrM$ where $\Fil ^r \scrM = \scrM \cap \Fil ^r \scrD$;
 \item $N_\scrD(\scrM) \subset \scrM$.
 \end{itemize}
\end{definition}
Assuming $r\le p-2$, Breuil constructs a functor $T_{\st}$ on the category
of strongly divisible lattices $\scrM$ in $\scrD$ with the property that
$T_{\st} (\scrM) \subset V$ is a $G_K$-stable $\Z_p$-lattice.
We refer the reader to \cite{Breuil-SDLattice} or \cite{LiuT-CofBreuil} for details.
The following theorem, synthesized from  \cite{Breuil-SDLattice}, \cite{KisinFcrystal}, and \cite{LiuT-CofBreuil},
summarizes the relations between Breuil--Kisin modules, strongly divisible $S$-lattices, and lattices in Galois representations:

\begin{theorem}\label{thm-classical} Let $V$ be a crystalline $G_K$-representation with Hodge-Tate weights in
$\{0, \dots, r\}$ and $T \subset V$ a $G_K$-stable $\Z_p$-lattice. Then
\begin{enumerate}
\item There is a unique filtered Breuil--Kisin module $M(T)$ of height $r$ with $\u T_\s (M(T)) \simeq T^\vee|_{G_\infty}$.\label{KisinExists}
\item There exists an $S[1/p]$-linear isomorphism $\alpha_S: \u \scrM (M(T)) [1/p]\simeq \scrD(V)$ which is compatible with $\varphi$ and filtrations.
\label{BreuilCompatibilityIsogeny}
\item If $r \leq p-2$, then the functor $T_{\st}$ induces an anti-equivalence between the category of strongly divisible $S$-lattices and the category of $G_K$-stable $\Z_p$-lattices $T $ in crystalline $G_K$-representations with Hodge-Tate weights in $\{0, \dots , r\}$.
\label{TstAntiEquiv}
\item In the situation of $(\ref{TstAntiEquiv})$, let
$\scrM(T)$ be the strongly divisible $S$-lattice with $T_{\st} (\scrM) \simeq T ^\vee$.
Then there is a natural isomorphism $\u \scrM (M(T)) \simeq \scrM (T)$ in $\Mod_S^{\varphi, r}$.
\label{BreuilCompatibility}
\end{enumerate}
\end{theorem}

\begin{proof}
Consider the version of the theorem with the classical Breuil--Kisin module $\m(T)$
in place of its filtered counterpart $M(T)$ and the functor $T_{\s}$
in place of $\u{T}_{\s}$.  In this scenario,
(\ref{KisinExists}) is proved in \cite{KisinFcrystal}, while (\ref{BreuilCompatibilityIsogeny}) is
proved in \cite{LiuT-CofBreuil}[\S 3.2].  We remark that the constructions of
the filtrations on $\scrD(V)$ and $\u{\scrM}(\m(T))$ are very different,
and that these two results have no restriction on $r$ and hold more generally
in the context of semistable $G_K$-representations.
Statements (\ref{TstAntiEquiv}) and (\ref{BreuilCompatibility}) of this variant of Theorem \ref{thm-classical}
are the main results of \cite{LiuT-CofBreuil}, and also hold more generally for semistable $V$.
Now by Lemma \ref{lem-dual}, we have $T_\s (\m(T)) \simeq \u T_\s (M(T))$ and $\u \scrM( M(T)) = \u\scrM (\m(T))$,
which, thanks to Remark \ref{ClassicalKisin}, then gives our version of the theorem.
\end{proof}

For future use, let us record a refinement of statements (\ref{KisinExists}) and
(\ref{BreuilCompatibilityIsogeny}) of Theorem \ref{thm-classical}.
Fix a $G_K$-stable $\Z_p$-lattice $T$ in a crystalline $G_K$-representation $V$
and for notational ease put $\m := \m(T)$, and $\scrD:= \scrD(V)$.
As in \cite{liu-Fontaine}[\S 3], one shows that $T_\s$ induces a natural injection
$$
A_{\inf} \otimes_{\s} \m \inj T^\vee \otimes _{\Z_p} A_{\inf}
$$
that intertwines $\varphi_{A_{\inf}}\otimes \varphi_{\m}$ with $\id\otimes \varphi_{A_{\inf}}$
and $g\otimes \id$ with $g\otimes g$ for $g\in G_{\infty}$.
Writing $M:=\varphi^*\m$ for the associated filtered Breuil--Kisin module,
we deduce from Lemma \ref{lem-dual} a similar injection
\begin{equation}\label{eqn-iota}
\iota_\s: A_{\inf} \otimes_{\s} M \inj T^\vee \otimes _{\Z_p} A_{\inf}
\end{equation}
that is likewise compatible with the actions of $\varphi$ and $G_{\infty}$.
The construction of the isomorphism $\alpha_S$ of (\ref{BreuilCompatibilityIsogeny}) given in \cite{LiuT-CofBreuil}
then shows that the following diagram is commutative:
\begin{equation}
\begin{gathered}
\xymatrix{  B^+_{\cris} \otimes _{S[1/p]}\u \scrM (M)[1/p]  \ar@{=}[r] \ar[d]^\wr_{B^+_{\cris} \otimes \alpha_S }  & B^+_{\cris} \otimes _\s M \ar@{^(->}[r]^{B^+_{\cris} \otimes \iota_\s} & T^\vee \otimes_{\Z_p} B^+_{\cris}\ar@{=}[d] \\
B^+_{\cris} \otimes_{S[1/p]} \scrD \ar@{=}[r] &  B^+_{\cris} \otimes _{W(k) [1/p]} D \ar@{^(->}[r]_-{\iota_D} & V^\vee \otimes_{\Z_p} B^+_{\cris}
}
\end{gathered}\label{CD:slphaS}
\end{equation}
where $\iota_D$ is the canonical injection arising from the very definition of
$D= D_{\cris} (V^\vee)$. Via this diagram, we henceforth regard
$\m \subset M \subset \scrM = \u \scrM (M)\subset \scrD$ as submodules of $B^+_{\cris} \otimes_{W(k)[1/p]} D$, which can be regarded as a submodule of $V^\vee \otimes_{\Z_p} B^+_{\cris}$ via $\iota_D$.

For use in the following section, we close this discussion with a brief
review of {\em Breuil-Kisin-Fargues} modules, adapted from \cite[\S4.3]{SBM}.
Let $F$ denote the fraction field of $R$.
\begin{definition}\label{BKFDef}
A \emph{Breuil--Kisin--Fargues module} is a pair $(\wt{M},\varphi_{\wt{M}})$
where $\wt{M}$ is a finitely presented $A_{\inf}$-module with the property
that $\wt{M}[1/p]$ is free over $A_{\inf}[1/p]$ and
$\varphi_{\wt M} : \wt M[\frac{1}{E(u)}] \simeq \wt M[\frac {1}{\varphi(E(u))}]$
is a $\varphi_{A_{\inf}}$-semilinear isomorphism.
Morphisms of Breuil--Kisin--Fargues modules are $\varphi$-compatible
$A_{\inf}$-module homomorphisms.
\end{definition}

Functorially associated to any Breuil--Kisin--Fargues module $(\wt{M},\varphi_{\wt{M}})$
is a pair $(L,\Xi)$ given by
$$
L := (\wt M \otimes_{A_{\inf}} W(F)) ^{\varphi=1}\quad \text{and}\quad \Xi := \wt M \otimes_{A_{\inf} } B^+_{\dR} \subset L \otimes_{\Z_p} B_{\dR}.  $$
One proves (see \cite{ScholzeWeinstein} and the discussion in \cite[\S4.3]{SBM})
that $L$ is a finite free $\Z_p$-module and $\Xi$ is a
$B_{\dR}^+$-lattice inside $L\otimes_{\Z_p} B_{\dR}$,
and that the functor
$(\wt{M},\varphi_{\wt{M}})\rightsquigarrow (L,\Xi)$
is an equivalence between the category of finite free Breuil--Kisin--Fargues modules
and the category of such pairs $(L,\Xi)$.

Now let $V$ be a crystalline $G_K$-representation with Hodge--Tate weights in $\{0,\ldots, r\}$,
and let $T\subset V$ be a  $G_K$-stable $\Z_p$-lattice.  Let $M(T)$ be the filtered Breuil--Kisin
module associated to $T$ as in Theorem  \ref{thm-classical}.

\begin{corollary}\label{cor-BKF} $A_{\inf} \otimes _\s M(T)$ is the Breuil-Kisin-Fargues module corresponding to
the pair $(T^\vee , \Xi)$, where $\Xi := M(T) \otimes_\s B^+_{\dR} = D \otimes_{W(k)[1/p]} B^+_{\dR}  \subset T^\vee \otimes_{\Z_p} B_{\dR}$.
\end{corollary}

\begin{proof}
	This is essentially \cite{SBM}[Prop. 4.34], except that we use the contravariant functors
	$\u{T}_\s$   and $T_\s$ from (filtered) Breuil--Kisin modules to Galois lattices.  It is straightforward to
	translate between the version in {\em loc.~cit.} and ours, as follows:
	It is clear that $\wt M := A_{\inf} \otimes_{\s} M (T) $ is a Breuil-Kisin-Fargues module.
	By \cite{liu-Fontaine} [Thm 3.2.2], the cokernel of the map
	$\iota_\s$ in \eqref{eqn-iota} is killed by $\varphi(\mathfrak t)^r$,
	where $\mathfrak{t}$ is a certain element of $W(R)$ satisfying $\varphi(\mathfrak{t})=E\mathfrak{t}$
	(we note here that our map $\iota_\s$ is the $\varphi$-twist of the map $\hat \iota$ in \cite[3.2.1]{liu-Fontaine}). Since $\mathfrak t$ is a unit of $W(F)$, we conclude that the scalar extension $W(F) \otimes_{A_{\inf}}{\iota_\s}: W(F) \otimes _\s M \to T^\vee \otimes_{\Z_p}W(F) $ is indeed an isomorphism. Passing to $\varphi$-invariants on both sides, we
	arrive at an isomorphism
$T^\vee = (W(F) \otimes_{A_{\inf}}\wt M)^{\varphi =1}$.
\end{proof}

\section{Crystalline cohomology}\label{pdivGone}

Let $\X$ be a  smooth and proper formal scheme over $\O_K$ with (rigid analytic) generic fiber $X=\X_K$ over $K$,
and put $\X_0:=\X\times_{\O_K} \O_K/(p)$ and $\X_k : = \X \times_{\O_K} k$.
For each nonnegative integer $i$, define
\begin{equation*}
	\scrM^i:=H^i_{\cris}(\X_0/S)\ \text{and}\ \scrD^i:=\scrM^i[1/p]
\end{equation*}
which are naturally $S$ and $S[1/p]$ modules, respectively, that are each
equipped with a semilinear Frobenius endomorphism $\varphi$.

Let $V^i := H^i_{\et} (X_{\o{K}}, \Q_p)$. By \cite[Theorem 1.1]{SBM},
$V^i $ is a crystalline $G_K$-representation with Hodge-Tate weights
in $\{-i , \dots , 0\}$.
Write $D^i := D_{\cris} ( (V^i)^{\vee})$  for the filtered $(\varphi,N)$-module associated to
the dual of $V^i$; of course, $N_D=0$ as $V^i$ is crystalline.
By the $C_{\cris}$ comparison theorem  \cite[Theorem 1.1]{SBM}, we have $D^i\simeq H^i _{\cris} (\X_k/ W(k))[1/p]$, compatibly with $\varphi$-actions and with filtrations after extending scalars to $K$.
Let $\wt \scrD ^i = \scrD ((V^i) ^\vee)$ be the filtered $(\varphi, N)$-module over $S$ attached to $(V^i) ^\vee$ as above Definition \ref{Def-SDlattice}.

Consider the natural projection $q: S\onto W(k)$ defined by $q (f(u))= f(0)$ for $f(u) \in S$. This induces
a natural map $\scrM^i \to H ^i_{\cris} (\X_k/ W(k))$ which we again denote by $q$.

\begin{proposition} \label{Prop-iso-formal}
There is a unique section $s: H ^i_{\cris} (\X_k/ W(k))[1/p ] \to \scrD^i= H^i_{\cris} (\X_0/ S )[1/p]$ of $q[1/p]$ satisfying
\begin{enumerate}
\item $s$ is $\varphi$-equivariant;
\item The map $ S \otimes_{W(k)}H ^i_{\cris} (\X_k/ W(k))[1/p ] \to \scrD^i$ induced by $s$ is an isomorphism.
\label{sectionisom}
\end{enumerate}
Identifying  $H ^i_{\cris} (\X_k/ W(k))[1/p ] = D^i$ then gives a $\varphi$-equivariant isomorphism
\begin{equation}
	\scrD^i \simeq S\otimes_{W(k)} D ^i =:\wt{\scrD}^i.
	\label{DDtileq}
\end{equation}
\end{proposition}
\begin{remark} The Proposition follows from \cite[Prop. (1.13.1)]{Beilin} (we are grateful to Koshikawa for pointing this out to us). The Proposition is also known when $\X$ is a  smooth proper {\em scheme} thanks to \cite[Lemma 5.2]{Hyodo-Kato} ({\em cf.} \cite[Prop.~4.4.6]{Tsuji}). For the sake of completeness, we include the
following argument here, which was suggested to us by Y.~Tian.
\end{remark}

\begin{proof}[Proof of Proposition $\ref{Prop-iso-formal}$]
The following  is a variant of the proof of \cite[Prop. 13.9]{SBM} obtained by replacing $A_{\cris}$ with $S$
and making some necessary modifications to the argument.
Let $S_{(n)}$ be the  $p$-adic completion of the PD-envelope of $W(k)[u_n]$ with respect of $(E(u_n) )$; note that $S_{(n)}$ and $S_n$ are different if $n>0$.  There is an evident inclusion $S\hookrightarrow S_{(n)}$,
and the Frobenius on $S$ uniquely extends to $\varphi:S_{(n)}\rightarrow S_{(n)}$.
Moreover, the self-map $\varphi^{n}$ of $S_{(n)}$ induces a $W(k)$-semilinear isomorphism
$\varphi^{n}:S_{(n)}\simeq S$.
Consider the projection $S_{(n)} \onto \O_{K_n}/ (\pi^e _n)$ given by $u_n \to \pi_n$.
This is a PD-thickening, and the isomorphism $\varphi^n: S_{(n )} \simeq S$ is compatible with the isomorphism
$\varphi^n : \O_{K_n}/ (\pi_n ^e) \simeq \O_{K}/(\pi^e) = \O_K / (p)$ sending $x$ to $x^{p^n}$,
so $\varphi^n:S_{(n)}\simeq S$ is a morphism of divided power thickenings.
Writing $\X_{(n)} : = \X \times_{\O_K} \O_{K_n}/(\pi_n^e)$, we thus have the following isomorphisms by base change
for crystalline cohomology:
\begin{equation}
	H^i _{\cris}(\X_{(n)}/ S_{(n)}) \otimes_{S_{(n)}, \varphi ^n} S \simeq H^i _{\cris} ( \X_{(n)} \times_{\O_{K_n}/ (\pi_n^e), \varphi^n} \O_K/ (p)/ S) \simeq H^i _{\cris} (\X _0/ S) \otimes_{S , \varphi^n } S.
	\label{BC1}
\end{equation}	
On the other hand, if $n$ is large enough (any $n$ with $p^n\ge e$ will do), we
have $\X_{(n)}\simeq \X_k \times_k \O_{K_n}/ (\pi_n^e)$ because the canonical map $\O_K\rightarrow \O_{K_n}/(\pi_n^e)$
factors through $k=\O_K/(\pi)$.  For such $n$, the inclusion $W(k)\hookrightarrow S_{(n)}$
is then a PD-morphism over $k\rightarrow \O_{K_n}/(\pi_n^e)$ so by base change and (\ref{BC1})
we find
\begin{equation}\label{eqn-isocrys}
 H^i_{\cris} (\X_k/ W(k)) \otimes_{W(k), \varphi^n} S \simeq H^i _{\cris}(\X_{(n)}/ S_{(n)}) \otimes_{S_{(n)}, \varphi ^n} S \simeq H^i _{\cris} (\X _0/ S) \otimes_{S , \varphi^n } S.
\end{equation}
Composing (\ref{eqn-isocrys}) with the map $\varphi^n\otimes 1 : H^i _{\cris} (\X _0/ S) \otimes_{S , \varphi^n } S  \to H^i _{\cris} (\X _0/ S)$, we obtain a map $s_n : H^i_{\cris} (\X_k/ W(k))  \to H^i_{\cris} (\X/S) $ that is $\varphi$-equivariant  and has the property that $q\circ s_n$ is simply
$\varphi ^n : H ^ i_{\cris} (\X_k / W(k)) \to H^i _{\cris} (\X_k/ W(k) ). $  Since $\varphi^n : H ^ i_{\cris} (\X_k / W(k))[1/p] \to H^i _{\cris} (\X_k/ W(k) ) [1/p]$ is bijective, we may finally define
$s:= s_n[1/p] \circ \varphi^{-n}: H_{\cris} ^i (\X_k / W(k)) [1/p] \to \scrD^i $, which by construction
is a $\varphi$-equivariant section of $q[1/p]$.

To show that the map $s(D^i)\otimes_{W(k)} S \to \scrD^i$ is bijective as claimed in (\ref{sectionisom}),
it suffices to show that $\varphi ^n \otimes 1 : \scrD ^i \otimes_{S, \varphi ^n} S \to \scrD^i$ is bijective. Since the identification $D \otimes_{W(k), \varphi^n } S \simeq \scrD^i \otimes_{S, \varphi^n }S$ of \eqref{eqn-isocrys} is compatible with $\varphi$,  and the map $\varphi^n \otimes \varphi ^n : D \otimes_{W(k), \varphi ^n} S \to D \otimes_{W(k), \varphi ^n} \varphi ^n(S) $ is bijective, it follows that
\begin{equation}
	\xymatrix{\varphi^n \otimes \varphi^n: \scrD ^ i \otimes _{S, \varphi ^n} S  \ar[r]^-{ \varphi ^ n \otimes 1 } & \scrD ^i \otimes_{S} S \ar[r] ^-{1 \otimes \varphi^n} &   \scrD ^i \otimes_{S, \varphi^n} \varphi ^n (S)}
	\label{phitensphi}
\end{equation} is a bijection as well. It is obvious that the second map in (\ref{phitensphi}) is bijective, because $\varphi: S \to \varphi(S)$ is an ring isomorphism. We conclude that $\varphi ^n \otimes 1 : \scrD ^i \otimes_{S, \varphi ^n} S \to \scrD^i$ is bijective, as desired.

That $s$ is unique is standard: if there exists another $\varphi$-equivariant section $s'$, then for any $x \in D^i$ we have $(s-s') (x) \in \Ker(q) \scrD ^i $. As $\varphi: D^i \to D^i$ is a bijection,
for all $m>0$
we may then write $x = \varphi ^m (y_m)$ for suitable $y_m$. But
then $$(s-s') (x)= (s-s') (\varphi ^m (y_m ) )= \varphi ^m ((s-s')(y_m)) \in  \varphi ^m(\Ker (q)) \scrD^i,$$
and this forces $(s-s') (x)= 0$.
\end{proof}

Using Proposition \ref{Prop-iso-formal},
we henceforth identify $\scrD^i $ with $\wt \scrD^i$ via (\ref{DDtileq}).
If we further assume that $\scrM^i$ is torsion free,
we may view $\scrM^i$ as an $S$-submodule of $\wt{\scrD}^i$, and
we then {\em define} $\Fil ^ i \scrM ^i : = \Fil ^i \wt \scrD^i \cap \scrM^i.$

\begin{remark}\label{rmk-1}
Note that the filtration $\Fil ^ i (\scrM ^i[1/p])$ is defined by Breuil's construction \eqref{eqn-fil}
using the Hodge filtration on de Rham cohomology and
the comparison isomorphisms
$$
H^i _{\cris} (\X_0/S) [1/p] \simeq S[1/p] \otimes_{W(k)} H^i_{\cris} (\X_k /W(k))
\quad\text{and}\quad
K \otimes_{W(k)}H^i_{\cris} (\X_k /W(k)) \simeq H^i_{\text{dR}} (X/K).
$$
It is natural to ask for a more direct definition of $\Fil ^ i (\scrM ^i[1/p])$ or even $\Fil^i \scrM ^i $
via the crystalline cohomology over $S$ of certain sheaves on the crystalline site of $\X_0$.
Such an interpretation will be explained in \S \ref{subsec-1}.
\end{remark}

Put $T^i:= H^i_{\et } (X_{\bar K}, \Z_p)/ \tor $; it is a $G_K$-stable $\Z_p$-lattice inside
the crystalline representation $V^i$ so has an associated filtered Breuil--Kisin
module $M^i:=M((T^i)^{\vee})$ via Theorem \ref{thm-classical} (\ref{KisinExists}).
Then Theorem \ref{thm-classical} (\ref{BreuilCompatibilityIsogeny})
provides
an isomorphism of $S[1/p]$-modules
$$ \alpha_S: \underline \scrM (M^i)[1/p] \simeq \wt \scrD^i $$
that is compatible  $\varphi$ and $\Fil^i$ (we forget the $N$-structure here).
Since we have identified $\wt \scrD^i $ with $\scrD^i $ in the above,  we arrive an isomorphism of $S[1/p]$-modules
\begin{equation}\label{eqn-iso}
\iota:  \underline \scrM (M ^i)[1/p] \simeq  \scrD^i = \scrM^i  [1/p]\simeq S\otimes_{W(k)} D^i
 \end{equation}
  compatible with $\varphi$ and $\Fil ^i$.
We reiterate that, as explained below (\ref{CD:slphaS}),
  we will  regard  all modules  involved in this discussion as submodules of
  $B^+_{\cris} \otimes_{W(k)[1/p]} D^i$ via $\iota$.

Our aim is to prove:

\begin{theorem}\label{thm-main-crystokisin}
Let $i$ be a nonnegative integer with $i <p-1$,
and assume that $H^i_{\cris} (\X_k/W(k))$ and $H^{i+1}_{\cris} (\X_k /W(k))$
are torsion free.  Then the following hold:
\begin{enumerate}
\item $T^i:=H^i_{\et} (X_{\o{K}}, \Z_p)$ is also torsion-free.\label{MTtorfree}
\item $\scrM^i:=H^i_{\cris}(\X_0/S)$ is a strongly divisible $S$-lattice in $\scrD ^i$ and
$T_{\st} (\scrM^i) \simeq T^i$.\label{MTSD}
\item  There is a natural isomorphism of Breuil--Kisin modules $\underline M (\scrM^i)\simeq M((T^i)^{\vee})$.
\label{MTBK}
\end{enumerate}
\end{theorem}

\begin{remark}\label{rmk-pre-results}
Assertions (\ref{MTtorfree}) and (\ref{MTSD})
of Theorem \ref{thm-main-crystokisin} have a long history, with many partial results.
When $K=K_0$ these facts were proved in \cite{FontaineMessing}.
Under the assumption $ei< p-1$, they follow from
\cite{CarusoInvent}, which proves a stronger comparison isomorphism at the level of
torsion objects; in these cases one can drop the assumption that
$H^j_{\cris} (\X_k/W(k))$
is torsion free for $j=i,i+1$ by replacing $T^i$ and $\scrM ^i$ with $T^i/\tor$ and $\scrM^i /\tor$, respectively,
and Theorem \ref{thm-main-crystokisin} (including assertion (\ref{MTBK})) still holds.
For arbitrary $e$, if one assume that $H^j_{\cris} (\X_k/W(k))$ is torsion free for \emph{all} $j$,
the work of Faltings \cite{Faltings} implies that (\ref{MTtorfree}), (\ref{MTSD}), and hence
(\ref{MTBK}) hold (though his terminology is slightly different from ours).
\end{remark}

To prove Theorem \ref{thm-main-crystokisin}, we must first recall \cite[Thm 1.8]{SBM}, which provides
a perfect complex $\RG_{A_{\inf}}(\X)$ of $A_{\inf}$-modules with a $\varphi$-linear map
$\varphi:\RG_{A_{\inf}}(\X)\rightarrow \RG_{A_{\inf}}(\X)$ such that
\begin{enumerate}

\item $\wt M^i := H^i (\RG_{A_{\inf}}(\X))$ is a Breuil--Kisin--Fargues module.
\label{BKFExists}
\item $H^i(\RG_{A_{\inf}}(\X) \otimes^{\mathbb L}_{A_{\inf}} A_{\cris}) \simeq
H^i _{\cris}(\X_{\O_{\o{K}}/(p)}/A_{\cris})\simeq A_{\cris} \otimes_{S}\scrM ^i$ as $\varphi$-modules over $A_{\cris}$.
\label{DerivedCrys}
\item $H^i(\RG_{A_{\inf}}(\X) \otimes ^{\mathbb L}_{A_{\inf}} W(\bar k))\simeq H^i _{\cris} (\X_{\o{k}}/W(\bar k))$
as $\varphi$-modules over $W(\bar k)$.
\label{DerivedW}
\item $H^i(\RG_{A_{\inf}}(\X) \otimes_{A_{\inf}}  W(F))\simeq H^i _{\et } (X_{\o{K}}, \Z_p)\otimes_{\Z_p} W(F)$
\label{Comparison:etale}
 \end{enumerate}
We advise the reader that (\ref{Comparison:etale}) is slightly different
from the comparison isomorphism found in \cite{SBM}, where the period
ring $A_{\inf}[1/\mu]$ is used in place of $W(F)$.
Here $\mu = [\varepsilon]-1$ for $\varepsilon = (\zeta_{p^n})_{n \geq 0} \in R$ with $\{\zeta_{p^n}\}$
a compatible system of primitive $p^n$-th root of unity.
However, it is not difficult to see that
$W(F)$ is flat over $A_{\inf}[1/\mu]$
({\em cf.} the proof of Lemma \ref{lem-prepare}) and then
(\ref{Comparison:etale}) follows easily from the comparison isomorphism found in \cite[Theorem 1.8]{SBM}.

From these facts we deduce a natural map of $\varphi$-modules over $A_{\cris}$:
\begin{equation}
	\tilde \iota   :  A_{\cris} \otimes_{A_{\inf}} \wt  M ^i \to H ^i _{\cris} (\X_{\O_{\o{K}}/(p)}/ A_{\cris})
	\label{map:tildeiota}
\end{equation}

 \begin{lemma}\label{lem-compatible[1/p]} There is a natural, $\varphi$-compatible
 isomorphism of $A_{\inf}$-modules
 $$\alpha: A_{\inf} \otimes_{ \s} M^i[1/p] \simeq \wt M ^i [1/p] $$
 with the property that the following diagram commutes:
 	\begin{equation}
	 \begin{gathered}
		\xymatrix{
		B^+_{\cris} \otimes_ { A_{{\inf}} }\wt M ^i \ar[r]^-{\tilde \iota[1/p]} & H^i_{\cris}(\X_{\O_{\o{K}}/(p)}/A_{\cris})[1/p]  \\
		B^+_{\cris} \otimes_{ \s}M ^i \ar[r]^-{\sim}_-{B^+_{\cris} \otimes \iota }  \ar[u]^{B^+_{\cris}\otimes\alpha }&  B^+_{\cris} \otimes_{S}  \scrM ^i\ar[u]^\wr}
		\label{Lemma:commdiag}
		\end{gathered}
	\end{equation}
In particular,  $\wt \iota [1/p]$ is an isomorphism.
   \end{lemma}

 \begin{proof}
 This follows from the main result of \cite{SBM} cited above.
 Indeed, by Definition \ref{BKFDef} and (\ref{BKFExists}) above,
 each $\wt{M}^i[1/p]$ is a finite and {\em free} $A_{\inf}[1/p]$-module,
so the derived comparison isomorphisms (\ref{DerivedCrys}) and (\ref{DerivedW})
yield comparison isomorphisms on the individual cohomology groups with $p$-inverted.
In particular, the natural induced maps
 $$
 \xymatrix{
 	{\tilde \iota[1/p]:   A_{\cris} \otimes_{A_{\inf}} \wt  M ^i[1/p]} \ar[r] &
	{H ^i _{\cris} (\X_{\O_{\o{K}}/(p)}/ A_{\cris})[1/p]}
	}
$$
and
$$
 \xymatrix{
 	{\wt M^i [1/p] \otimes_{A_{\inf}} W(\bar k)}\ar[r] & {H^i_{\cris} (\X_{\o{k}}/ W(\bar k))[1/p] = W(\bar k) \otimes_{ W(k)}  D^i}
}	
$$
are $\varphi$-compatible isomorphisms.
Furthermore, as explained following \cite[Thm. 1.8]{SBM}, these mappings
are compatible with the canonical projection
$q: \wt M^i [1/p] \to  \wt M^i [1/p] \otimes_{A_{\inf}} W(\bar k)$
and the projection $ q' :  H ^i _{\cris} (\X_{\O_{\o{K}}/(p)}/ A_{\cris})[1/p] \to H^i_{\cris} (\X_{\o{k}}/ W(\bar k))[1/p] $
arising via the compatibility of crystalline cohomology with PD-base change
in the sense that the obvious diagram commutes.
It follows that $\tilde \iota$ induces a $\varphi$-compatible  isomorphism
$$
\xymatrix{
	A_{\cris} \otimes_{A_{\inf}} \wt  M ^i[1/p] \ar[r]_-{\wt{\iota}[1/p]}^-{\sim} &
	H ^i _{\cris} (\X_{\O_{\o{K}}/(p)}/ A_{\cris})[1/p] &
	B^+_{\cris} \otimes_{W(k)[1/p]} \bar D^i  \ar[l]^-{\beta}_-{\sim}
	}
$$
where $\bar D^i := H^i_{\cris} (\X_{\o{k}}/ W(\bar k))[1/p]$ and the isomorphism $\beta$ is induced by constructing a
$\varphi$-equivariant section $s: \bar D^i  \inj H ^i _{\cris} (\X_{\O_{\o{K}}/(p)}/ A_{\cris})[1/p]$
to the projection map $q'$; we note that such a section $s$ exists and is necessarily unique,
whence $\beta$ is unique as well (see \cite[Prop. 13.9]{SBM} and Proposition \ref{Prop-iso-formal}).
Using the isomorphisms $\tilde \iota[1/p]$ and $\beta$,  we then regard $\wt M^i [1/p] $ and
$H^i_{\cris} (\X_{\O_{\o{K}}/(p)}/ A_{\cris})$ as submodules of $B^+_{\cris} \otimes_{W(k)[1/p]} D^i$.
As explained in the beginning of this section, we use the map $\iota$ of \eqref{eqn-iso} to regard both
$M^i$ and $\scrD ^i = H^i _{\cris}( \X_0/ S)[1/p]$ as submodules of $B^+_{\cris} \otimes_{W(k)[1/p]} D^i$.
Thus, working entirely inside $B^+_{\cris} \otimes_{W(k)[1/p]} D^i$, it now suffices to prove
that $ A_{\inf} \otimes_{\s}M^i[1/p] = \wt M^i[1/p]$.
By \cite{SBM}[Prop. 4.13], there exists an exact sequence of $A_{\inf}$-modules
\begin{equation}
	\xymatrix{
		0 \ar[r] & {\wt M ^ i_{\tor}} \ar[r] & {\wt M^i} \ar[r] & {\wt M^i_{\free}}\ar[r] & {\bar M^i} \ar[r] & 0.
		}\label{BKstructure}
\end{equation}
where $\wt M^i_{\tor}$ is killed by power of $p$, the term $\wt M^i _{\free}$ free of finite rank over $A_{\inf}$,
and $\bar M^i$ is killed by some power of the ideal $(u,p)$.
We claim that $\wt M^i_{\free}$ is the Breuil--Kisin--Fargues module corresponding
to the pair $(T^i,\Xi)$, with $\Xi := B^+_{\dR} \otimes_{A_{\inf}}\wt M^i_{\free} \subset T^i \otimes_{\Z_p}  B_{\dR}. $
To see this, we apply $\otimes_{A_{\inf}} W(F)$ to the exact sequence (\ref{BKstructure})
and, using Lemma \ref{lem-prepare} \eqref{isflat} below with the fact that $\o{M}^i$
is killed by a power of $u\in W(F)^{\times}$, we deduce an exact sequence of $W(F)$-modules
\begin{equation}
	\xymatrix{
		0 \ar[r] & {\wt{M}_{\tor}^i \otimes_{A_{\inf}}W(F)}  \ar[r] & {\wt{M}^i \otimes_{A_{\inf}}W(F)} \ar[r] &
		{\wt{M}_{\free}^i \otimes_{A_{\inf}}W(F)} \ar[r] & 0
	}.\label{ExSeq:WF}
\end{equation}
Writing $T_{\star}:=(\wt{M}_{\star}^i \otimes_{A_{\inf}}W(F))^{\varphi=1}$ for $\star\in \{\tor,\emptyset,\free\}$,
Lemma 4.26 of \cite{SBM} gives that $T_{\star}$ is a finite-type $\Z_p$-module and canonically identifies (\ref{ExSeq:WF}) with the exact sequence
\begin{equation}
	\xymatrix{
		0 \ar[r] & {T_{\tor} \otimes_{A_{\inf}}W(F)}  \ar[r] & {T \otimes_{A_{\inf}}W(F)} \ar[r] &
		{T_{\free} \otimes_{A_{\inf}}W(F)} \ar[r] & 0
	}.\label{ExSeq:WF2}
\end{equation}
From the very definition of $T_{\star}$, we have a sequence of $\Z_p$-modules
\begin{equation}
	\xymatrix{
		0 \ar[r] & T_{\tor} \ar[r] & T \ar[r] & T_{\free} \ar[r] & 0
	},\label{Seq:T}
\end{equation}
whose scalar extension to $W(F)$ is the exact sequence (\ref{ExSeq:WF2}).  Thus, since $W(F)$
is faithfully flat over $\Z_p$ \cite[Tag 0539]{stacks-project}, we conclude that (\ref{Seq:T}) is {\em exact}.
Now since $\wt{M}_{\free}^i$ if finite free over $A_{\inf}$, it is clear that $T_{\free}$
is free over $\Z_p$, and we have $T\simeq H^i_{\et}(X_{\o{X}},\Z_p)$ thanks to
the comparison isomorphism $H^i(\RG_{A_{\inf}}(\X) \otimes_{A_{\inf}}  W(F))\simeq H^i _{\et } (X_{\o{K}}, \Z_p)\otimes_{\Z_p} W(F)$ of \cite[Theorem 1.8]{SBM} recorded above and the fact that $A_{\inf}\rightarrow W(F)$
is flat, recorded in Lemma \ref{lem-prepare} \eqref{isflat} below.
It follows at last that we have
$(\wt{M}^i_{\free}\otimes_{A_{\inf}}W(F))^{\varphi=1}=: T = H^i_{\et}(X_{\o{K}},\Z_p)/\tors =: T^i$, which gives our claim.
Now since we clearly have $\wt M^i _{\free}[1/p] = \wt M^i [1/p]$, we may rewrite $\Xi = B^+_{\dR} \otimes_{A_{\inf}}\wt M^i_{\free} = B^+_{\dR} \otimes_{W(k)[1/p]} D^i$.  Since $M^i = M((T^i) ^\vee)$, Corollary \ref{cor-BKF} then shows that $ A_{\inf} \otimes_{\s} M^i$ is the Breuil--Kisin--Fargues module corresponding to $(T^i , \Xi)$.
This yields the desired identification
$A_{\inf} \otimes_{\s} M^i= \wt M^i _{\free}$ inside  $B^+_{\cris} \otimes_{W(k)[1/p]} D^i$.
\end{proof}

To proceed further, we will need:

 \begin{proposition}\label{prop-isotiliota} Suppose that $\wt M ^{i+1}$ is $u$-torsion free.  Then
 $(\ref{map:tildeiota})$ is an isomorphism.
 \end{proposition}

To prove this proposition, we first require some preparations.
Put $F:=\Frac(R)$ and note that $W(F)$ is a complete DVR with uniformizer $p$.
 \begin{lemma} \label{lem-prepare}
 The following statements hold:
 \begin{enumerate}
 \item $W(F)$ is flat over $A_{\inf}$.\label{isflat}
 \item  $W(F) \cap A_{\inf}[1/p] = A_{\inf}$.\label{intersect}
\item Let  $M$ be a finitely presented $A_{\inf}$-module with $M[\frac 1 p]$ finite and free over $A_{\inf}[\frac 1 p]$. 
If $M$ is $u$-torsion free then the map $M\to W(F) \otimes_{A_{\inf}} M $ is injective.
\label{noutor}
\item $A_{\cris} / (p^n)$ is faithfully flat over $S/(p^n)$ for all $n\ge 1$.
\label{fflat}
\end{enumerate}
 \end{lemma}
\begin{proof} Let $A' := ({A_{\inf}})_{(p)}$ be the localization of $A_{\inf}$ at the prime ideal $(p)$. Then $A'$
is a local ring with uniformizer $p$ and residue field $F=\Frac R$, so is a discrete valuation ring and in particular is
noetherian.  As a localization of $A_{\inf}$, it is moreover flat over $A_{\inf}$.
One checks via the theory of strict $p$-rings that $W(F)$ is the $p$-adic completion of $A'$, and hence flat over $A'$
as the completion of {\em any} noetherian local ring is flat \cite[Tag 00MB]{stacks-project}.
It follows that $W(F)$ is flat over $A_{\inf}$ giving (\ref{isflat}).

To prove (\ref{intersect}), suppose $x \in A_{\inf} [1/p] \cap W(F)$.
Then for $m$ sufficiently large we have
$$p ^m x \in A_{\inf} \cap p ^m W(F)= p^m W(R) = p^m A_{\inf}$$
by basic properties of Witt vectors, which gives $x\in A_{\inf}$.

To prove (\ref{noutor}), we proceed as follows.
First, for ease of notation, if $N$ is any $A_{\inf}$-module, we will write $\iota_N$ for the natural map
$N\rightarrow N\otimes_{A_{\inf}} W(F)$.
As $M[\frac 1 p]$ is finite and free over $A_{\inf}[\frac 1 p]$, one has an exact sequence as in 
(\ref{BKstructure}), which may be split as two short exact sequences:
\begin{subequations}
\begin{equation}
	\xymatrix{
	0 \ar[r] &  M_{\rm tor} \ar[r] &  M \ar[r] &  M'\ar[r] &  0
	}\label{seqA}
\end{equation}
\begin{equation}	
	\xymatrix{
	0 \ar[r] & M' \ar[r] &  M_{\rm free} \ar[r] & {\overline M} \ar[r] & 0
	}\label{seqB}
\end{equation}
\end{subequations}
where $M_{\rm tor}$ is killed by $p^n$ for some $n$, and $M_{\rm free}$ is a finite free $A_{\inf}$-module. 
Since $W(F)$ is flat over $A_{\inf}$ by the already established (\ref{isflat}), the sequences (\ref{seqA}) and (\ref{seqB})
remain short exact after extending scalars to $W(F)$.  It follows that to prove $\iota_M$
is injective, it suffices to prove that $\iota_{M'}$ and $\iota_{M_{\rm tor}}$ are both injective.
Now $\iota_{M_{\rm free}}$ is certainly injective as $M_{\rm free}$ is free over $A_{\inf}$,
so $\iota_{M'}$ is injective thanks to (\ref{seqB}) and what we have observed above.
It remains to prove that $\iota_{M_{\rm tor}}$ is injective, so without loss of generality
we may reduce to the case that $M$ is killed by $p^n$ for some $n$.  We will show
that $\iota_M$ is injective by induction on $n$.  First suppose $n=1$.
Then $M$ is finitely presented as a module over the valuation ring $R=A_{\inf}/pA_{\inf}$, 
whence it is a direct sum of a finite free $R$-module and a module of the form
$\bigoplus_{i=1}^m R/a_i R$ with $a_i\in R$ nonzero \cite[Tag 0ASU]{stacks-project}; {\em cf.}~the proof of 
Lemma \ref{lem-injection} below.
 As $M$ is assumed to be $u$-torsion free, we must have $a_i\in R^{\times}$ for all $i$, 
 and $M$ is finite and free over $R$.  It follows at once that $\iota_M: M\rightarrow M\otimes_R F$ is injective.
Now suppose that $M$ is killed by $p^n$ for some $n>1$ and consider the exact sequence
\begin{equation}\label{eq1}
\xymatrix{
0 \ar[r] &  M[p] \ar[r] &  M  \ar[r] & M'  \ar[r] &  0, 
}
\end{equation}
where $M[p ]= \{x\in M\ :\ p x = 0 \}$. 
We claim that $M'$ has no $u$-torsion.  Indeed, if $y \in M$ has $u y \in M[p]$ then $ p u y = u (py) = 0$. 
Since $M$ has no $u$-torsion, this forces $py= 0$ and $y \in M[p]$, so that $M'$ is $u$-torsion free. 
Since $M[p]$ is also $u$-torsion free (being a submodule of a $u$-torsion free module)
and both $M[p]$ and $M'$ are killed by $p^{n-1}$, our inductive hypothesis gives that 
$\iota_{M[p]}$ and $\iota_{M'}$ are injective, and it follows that $\iota_M$ is injective as well.

To prove (4), let $S^{\PD}:= \s[\{E(u )^n/n!\}_{n \geq 1}]$ be the divided power envelope of $\s$
with respect to the kernel of
the surjection $\s\twoheadrightarrow \O_K$ sending $u$ to $\pi_0$, so that $S$ is the $p$-adic completion
of $S^{\PD}$.  We similarly write $A_{\inf}^{\PD}$ for the divided power envelope
of $A_{\inf}$ with respect to $\ker(\theta_{\inf})$, so again $A_{\cris}$ is the $p$-adic completion
of $A_{\inf}^{\PD}$.  We claim that the natural map $A_{\inf}\otimes_{\s} S^{\PD}\rightarrow A_{\inf}^{\PD}$
is an isomorphism.  Indeed, this follows from \cite[Tag 07HD]{stacks-project}
once we check that $\s/(p)\rightarrow A_{\inf}/(p)$ is flat
and $\Tor_1^{\s}(A_{\inf},\s/(p))=0$.  As $\s/(p)=k[\![u]\!]$ is a DVR and $A_{\inf}/(p)=R$ is torsion-free,
the first is clear \cite[Tag 0539]{stacks-project}, as is the second since $A_{\inf}=W(R)$ is $p$-torsion free.
Thanks to \cite{SBM}[Lem. 4.30], the map $\s\rightarrow A_{\inf}$
is flat, whence its scalar extension
$S^{\PD}= \s \otimes_{\s} S^{\PD}\rightarrow A_{\inf}\otimes_{\s}S_{\PD}=A^{\PD}_{\inf}$
is flat as well.  This implies that the map $S/(p^n)=S^{\PD}/(p^n)\rightarrow A^{\PD}_{\inf}/(p^n)=A_{\cris}/(p^n)$
is flat for every $n\ge 1$.  Since $S\rightarrow A_{\cris}$ is a local map of local rings
we conclude that the flat maps  $S/(p^n)\rightarrow A_{\cris}/(p^n)$ are faithfully flat.
\end{proof}

\begin{remark}
	Perhaps surprisingly, we do not know whether or not $S\rightarrow A_{\cris}$ is (faithfully) flat.
\end{remark}

In what follows, for an $A_{\inf}$-module $M$ we will simply write
$M_{W(F)}$ for the scalar extension $W(F) \otimes_{A_{\inf}}M$.
For a map $f:M\rightarrow M'$ of $A_{\inf}$-modules,
we write $f_{W(F)}:=f\otimes 1 : M_{W(F)}\rightarrow M'_{W(F)}$
for the induced map of $W(F)$-modules.

  \begin{lemma}\label{lem-injection} Let  $f : M \to M'$ be  a map of $A_{\inf}$-modules. Assume that
  \begin{enumerate}\item $M$ is finite and free over $A_{\inf}$,
  \item $M'$ is finitely presented over $A_{\inf}$ and $u$-torsion free,	
	with $M'[\frac 1 p]$  finite and free over $A_{\inf}[\frac 1 p]$.
  \item  $N:  = \Ker(f)$ is a finitely generated $A_{\inf}$-module.
 \end{enumerate}
   Then $N$ is finite free over $A_{\inf}$.
 \end{lemma}

 \begin{proof}  Lemma \ref{lem-prepare} gives a commutative diagram with exact rows
 $$ \xymatrix{ 0 \ar[r]  & N_{W(F)} \ar[r] & M_{W(F)} \ar[r] & M'_{W(F)}\\
 0 \ar[r]  & N \ar[r] \ar@{^(->}[u]  \ar@{^(->}[d]  & M \ar[r]  \ar@{^(->}[u]  \ar@{^(->}[d]   & M'   \ar@{^(->}[u]  \ar[d]    \\
 0 \ar[r]  & N[1/p] \ar[r] & M[1/p] \ar[r] & M'[1/p] 
 }
 $$
in which all vertical arrows, with the possible exception of the lower right arrow, are injective.
Now we claim that $N = N[1/ p] \cap N_{W(F)}$ inside $ M_{W(F)}[1/p]$.
To prove the claim,
let $x \in N[1/p]\cap N_{W(F)}  \subset M [1/p] \cap M_{W(F)}$ be arbitrary.
Then from the diagram we see that $f[1/p](x) = f_{W(F)} (x) = 0$.
On the other hand,
using Lemma \ref{lem-prepare} (\ref{intersect}) and our hypothesis that $M$ is finite free over $A_{\inf}$,
we deduce that $x \in M [1/p] \cap M_{W(F)}= M$.
Furthermore, since the upper vertical arrows are injective
and $f_{W(F)}( x)= 0$, we must have $f(x)= 0$ and $x\in \ker(f)= N$ as claimed.
From this claim it follows at once that the natural map $N/ p N\to  N_{W(F)}/ p N_{W(F)}$
is {\em injective}.  Since the target is a finite dimensional $F$-vector space
and the source is a finitely generated $R$-module, we conclude that $N/pN$
is a finite and {\em torsion-free} $R$-module, and hence a finitely generated submodule of
a free $R$-module.\footnote{Indeed, an elementary argument shows that any
finitely generated and torsion-free module over a commutative domain
may be embedded as a submodule of a free module.
}
Now $R$ is a (non-noetherian) valuation ring, and hence a B\'ezout domain,
from which it follows that $N/pN$ is a finite and {\em free} $R$-module \cite[Tag 0ASU]{stacks-project}.
Let $x_1, \dots , x_m$ be a $R$-basis of $N/ pN$ and choose lifts $\hat x_i \in N$.
By Nakayama's Lemma, these lifts generate $N$ as an $A_{\inf}$-module.
Now any nontrivial $A_{\inf}$-relation $\sum \alpha_i \hat{x}_i = 0$
on these generators may be re-written as $p^j\sum \alpha_i' \hat{x}_i = 0$
for some nonnegative integer $j$, with at least one $\alpha_i'$ nonzero
modulo $p$.  But $N$ is a submodule of the free module $M$, and hence
torsion-free, from which it follows that  $\sum \alpha_i' \hat{x}_i = 0$.
But this relation reduces modulo $p$ to a {\em non-trivial} relation
on the $x_i$, contradicting the fact that $N/pN$ is $R$-free.
We conclude that the $\hat{x}_i$ freely generate $N$, as desired.
 \end{proof}

\begin{proof} [Proof of Proposition $\ref{prop-isotiliota}$]
Let $C^\bullet$ be a bounded complex of finite projective $A_{\inf}$-modules
that is quasi-isomorphic to the perfect complex $\RG_{A_{\inf}}(\X)$.
As $A_{\inf}$ is a local ring, each term $C^i$ is a finite, free $A_{\inf}$-module.
We then have
$\RG_{A_{\inf}}(\X) \otimes^{\mathbb L}_{A_{\inf}} A_{\cris} = C^\bullet \otimes_{A_{\inf}} A_{\cris}$,
and our goal is to prove that the canonical map
$$H^i(C^\bullet)\otimes_{A_{\inf}} A_{\cris} \rightarrow H^i(C^\bullet \otimes_{A_{\inf}} A_{\cris})$$
is an isomorphism.  Writing $d^j: C^j\to C^{j+1}$ for the given maps,
we first show that the natural map
\begin{equation}
	A_{\cris}    \otimes_{A_{\inf}}  \Ker (d^i ) \to \Ker (A_{\cris}  \otimes_{A_{\inf}}d^i )
	\label{map:prelim}
\end{equation}	
is an isomorphism.	
	 To see this, consider the exact sequences of $A_{\inf}$-modules
\begin{subequations}
	\begin{equation}
		\xymatrix{
	0 \ar[r] & \ker(d^i) \ar[r] & C^i \ar[r] & \im(d^i) \ar[r] & 0
	}\label{seq:kerd}
	\end{equation}	
	\begin{equation}
		\xymatrix{
	0 \ar[r] & \im(d^i) \ar[r] & \ker(d^{i +1}) \ar[r] & {\wt M^{i +1}} \ar[r] & 0
	}\label{seq:imd}
	\end{equation}	
\end{subequations}
Now $C^i$ is finite over $A_{\inf}$, so same is true of its quotient $\im(d^i)$,
whence $\ker(d^{i+1})$ is finitely generated because the Breuil--Kisin--Fargues module $\wt{M}^{i+1}$ is.
Likewise, $\ker(d^i)$ is finitely generated, and since $C^i$ and $C^{i+1}$
are finite free, it follows from Lemma \ref{lem-injection} that $\ker(d^i)$
and $\ker(d^{i+1})$ are finite free---hence flat---$A_{\inf}$-modules.
Since $A_{\inf}$ and $A_{\cris}$ are both contained in the field $B _{\dR}$,
an easy argument then shows that the inclusions
$\ker (d^i) \hookrightarrow C^i$ and $\ker (d^{i +1})\hookrightarrow  C^{i +1}$ remain
injective after tensoring with $A_{\cris}$; in particular, the sequence (\ref{seq:kerd})
remains exact after tensoring with $A_{\cris}$.

In general, the sequence (\ref{seq:imd}) may not
remain exact after tensoring with $A_{\cris}$.  However, using our hypothesis
that $\wt{M}^{i+1}$ is $u$-torsion free and the fact that $\im(d^i)$
is finitely generated, Lemma \ref{lem-injection} shows that
$\im(d^i)$ is in fact finite and free as an $A_{\inf}$-module.
Arguing as above, we conclude that the sequence (\ref{seq:imd})
likewise remains exact after tensoring with $A_{\cris}$.
Since $\ker (d^{i +1})\hookrightarrow  C^{i +1} $ remains injective
after tensoring with $A_{\cris}$, it follows that the
map $\im(d^i)\otimes_{A_{\inf}} A_{\cris}\to \im(d^i\otimes A_{\cris})$
is injective.  An easy diagram chase now shows that the
map (\ref{map:prelim}) is an isomorphism as claimed.

Now consider the following diagram
$$
\xymatrix{
&  A_{\cris} \otimes_{A_{\inf}} \im (d^{i-1}) \ar@{->>}[d] \ar[r] & A_{\cris} \otimes_{A_{\inf}} \Ker (d^i)\ar[d]^\wr   \ar[r] & A_{\cris} \otimes_{A_{\inf}} \wt M^i \ar[r]\ar[d]^{\wt{\iota}} &  0\\
0 \ar[r] &   \im ( A_{\cris} \otimes_{A_{\inf}} d^{i-1})  \ar[r] &  \Ker (A_{\cris} \otimes_{A_{\inf}}d^i)  \ar[r] &  H^i _{\cris} (\X/ A_{\cris})\ar[r]  & 0
}
$$
Since the first column is surjective by right-exactness of tensor product
and we have just seen that the
second column is an isomorphism, the Snake Lemma completes the proof
that the map $\wt{\iota}$ of the third column is an isomorphism.
\end{proof}

\begin{remark}
	Our application of Proposition \ref{prop-isotiliota} will be to the proof of
	Theorem \ref{thm-main-crystokisin}.  Under the hypotheses of
	this theorem, one knows that
	$\wt{M}^j$ is in fact {\em free} of finite rank over $A_{\inf}$ for $j=i, i+1$
	(see below), so for our purposes it would be enough to have the conclusion of
	Proposition \ref{prop-isotiliota} under the stronger hypothesis that $\wt{M}^{i+1}$
	is free over $A_{\inf}$.  The following short proof of this variant was suggested to us by Bhargav Bhatt.
	
	In the notation of the proof of Proposition \ref{prop-isotiliota},
	and putting $M:=\wt{M}^j=H^j(C^\bullet)$,
	we first claim that
	the complex $M \otimes^{\mathbb L}_{A_{\inf}} A_{\cris}$
	is concentrated in homological degrees 0 and 1; that is, that $\Tor_i^{A_{\inf}}(M, A_{\cris})=0$ for $i\ge 2$.
	To see this, first note that $M$ has bounded $p$-power torsion and, as a complex of $A_{\inf}$-modules,
	is perfect thanks to Theorem 1.8 and Lemma 4.9 of \cite{SBM}.  It follows that
	the pro-systems $\{M/p^nM\}_n$ and $\{M \otimes^{\mathbb L}_{A_{\inf}} A_{\inf}/(p^n)\}$ are pro-isomorphic,
	and that $M\otimes^{\mathbb L}_{A_{\inf}} A_{\cris}$ is $p$-adically complete.  We deduce isomorphisms
	\begin{equation*}
		M\otimes^{\mathbb L}_{A_{\inf}} A_{\cris} \simeq
		\varprojlim_n \left(M \otimes^{\mathbb L}_{A_{\inf}} A_{\cris}/(p^n)\right)\simeq
				\varprojlim_n \left(M/p^nM \otimes^{\mathbb L}_{A_{\inf}/(p^n)} A_{\cris}/(p^n)\right).
	\end{equation*}
	The claim follows from the fact that $A_{\cris}/(p^n)$ has Tor-dimension 1 over $A_{\inf}/(p^n)$.\footnote{Indeed,
	one has the evident presentation
	$\xymatrix@1{A_{\inf}/(p^n)\langle T\rangle \ar@{^{(}->}[r]^-{T-\xi_n} & A_{\inf}/(p^n)\langle T\rangle \ar@{->>}[r]&
	A_{\cris}/(p^n) },$
	where $\xi_n$ is a generator of the principal ideal $\ker(A_{\inf}/(p^n)\twoheadrightarrow \O_{\Kbar}/(p^n))$.
	From the very construction of the divided power polynomial algebra $A_{\inf}/(p^n)\langle T\rangle$,
	 this is then a 2-term resolution by free---hence flat---modules, so \cite[Tag 066F]{stacks-project} gives the asserted Tor-dimension.
	}
	
	Next, we claim that the cohomology groups
	$H^{j}(\tau^{> i+1}C^\bullet  \otimes^{\mathbb L}_{A_{\inf}} A_{\cris}) $
	vanish for $j \le i$.  More generally, suppose that
	$N^{\bullet}$ is any bounded complex with $\tau^{< 0} N^{\bullet} = 0$ and
	$H^j(N^{\bullet}) \otimes^{\mathbb L}_{A_{\inf}} A_{\cris}$ concentrated in homological degrees
	0 and 1.  Then an easy induction argument on the number of nonzero terms in $N^{\bullet}$
	shows that $\tau^{<1}(N^{\bullet} \otimes^{\mathbb L}_{A_{\inf}} A_{\cris}) =0$,
	and hence that $H^0(N^{\bullet} \otimes^{\mathbb L}_{A_{\inf}} A_{\cris})$---which
	is a quotient of this complex---is zero as well.  Applying this with
	$N^{\bullet}=(\tau^{> i+1}C^\bullet)[-j]$ then gives the claimed vanishing.

	Applying $\otimes^{\mathbb L}_{A_{\inf}} A_{\cris}$ to the exact triangle
	\begin{equation*}
		\xymatrix{
			\tau^{\le i+1} C^\bullet \ar[r] & C^\bullet \ar[r] & \tau^{> i+1}C^\bullet
		}
	\end{equation*}
	and passing to the long exact sequence of cohomology modules thus yields an isomorphism
	\begin{equation*}
					H^{i}(\tau^{\le i+1}C^\bullet  \otimes^{\mathbb L}_{A_{\inf}} A_{\cris}) \simeq
			H^i(C^\bullet\otimes^{\mathbb L}_{A_{\inf}} A_{\cris}).
	\end{equation*}	
	On the other hand, applying $\otimes^{\mathbb L}_{A_{\inf}} A_{\cris}$ to the exact triangle
	\begin{equation*}
		\xymatrix{
			\tau^{\le i} C^\bullet \ar[r] & \tau^{\le i+1} C^\bullet \ar[r] & H^{i+1}(C^\bullet)[-i-1]
		}
	\end{equation*}
	and passing to cohomology gives the short exact sequence
	\begin{equation*}
		\xymatrix{
			0 \ar[r] & H^i(\tau^{\le i} C^\bullet \otimes^{\mathbb L}_{A_{\inf}} A_{\cris})
			\ar[r] & H^i(\tau^{\le i+1} C^\bullet \otimes^{\mathbb L}_{A_{\inf}} A_{\cris})
			\ar[r] & \Tor_1^{A_{\inf}}(H^{i+1}(C^\bullet),A_{\cris}) \ar[r] & 0
		}
	\end{equation*}
	in which the first term is readily seen to be isomorphic to $H^i(C^{\bullet})\otimes_{A_{\inf}}A_{\cris}$ since we are taking
	``top degree'' cohomology.  Thus, when $\wt{M}^{i+1}=H^{i+1}(C^\bullet)$ is free over $A_{\inf}$ so the Tor
	vanishes, we deduce that (\ref{map:tildeiota}) is an isomorphism, as desired.
\end{remark}

\begin{proof}[Proof of Theorem $\ref{thm-main-crystokisin}$]
	If $j$ is {\em any} nonnegative integer with the property that  $H^j_{\cris}(\X_k/W(k))$ is torsion free,
	then Theorem 14.5 and Proposition 4.34 of \cite{SBM} show
	 \begin{enumerate}
	 	 \item $T= H^j_{\et} (X_{\o{K}}, \Z_p)$ is finite free of  $\Z_p$-rank
		 $d:=\dim_{\Q_p}H^j_{\et}(X_{\o{K}},\Q_p)$.
	 	 \item $\wt M^j$ is finite free of rank $d$ over $A_{\inf}$ and $\wt M^j \simeq A_{\inf} \otimes_{\s} M^j$
		 via the map $\alpha$ of Lemma \ref{lem-compatible[1/p]}.
 \end{enumerate}
 	Thus, our hypothesis that $H^j_{\cris}(\X_k/W(k))$ is torsion-free for $j=i, i+1$
 	implies in particular that $\wt{M}^{i+1}$ is $u$-torsion free, and hence that the natural map
 	 $$
	 \xymatrix{
	 	{\tilde \iota : A_{\cris} \otimes _{A_{\inf}} \wt M^i}\ar[r] &
		{H^i (\X_{\O_{\o{K}}/(p)}/ A_{\cris}) = A_{\cris} \otimes_S \scrM^i}
		}
	$$	
	 of (\ref{map:tildeiota}) is an isomorphism thanks to Proposition \ref{prop-isotiliota}.
	Thus, $ A_{\cris} \otimes_S \scrM^i$ is a finite and free $A_{\cris}$-module of rank $d$,
	so also $A_{\cris} \otimes_S \scrM^i/(p^n)$ is finite and free of rank $d$ as an $A_{\cris}/(p^n)$-module.
	 Using Lemma \ref{lem-prepare} (\ref{fflat}) together with
	the facts that the property ``finite projective" for modules descends
	along faithfully flat morphisms \cite[Tag 058S]{stacks-project} and
	finite projective implies free \cite[Tag 00NZ]{stacks-project} (see also \cite[Tag 0593]{stacks-project}),		
	we deduce that $\scrM^i/(p^n)$ is finite and free over $S/(p^n)$, necessarily of rank $d$.
	Let $e_1, \dots , e_d$ be a basis of $\scrM^i/(p^n)$ and choose lifts $\hat e_j\in \scrM^i$ of $e_j$.
	By Nakayama's Lemma, $\scrM^i$ is generated by the $\hat e_j$.
	There can be no linear relations as $\scrM^i[1/p] = \scrD^i\simeq \wt{\scrD}^i = S\otimes_{W(k)} D^i$ is
	finite and free over $S[1/p]$ of rank $d=\dim_{W(k)[1/p]} D^i$. This proves that $\scrM^i$ is finite $S$-free.

	 Consider the isomorphism
	 $\iota:  \underline \scrM (M^i)[1/p] \simeq \scrM^i  [1/p]$
	of (\ref{eqn-iso}).  Identifying $\u{\scrM}(M^i)$ with its image
	in $\scrD^i=\scrM^i[1/p]$ under $\iota$, our goal is then to prove that
	$\u{\scrM}(M^i)=\scrM^i$ inside $\scrD^i$, and to do so 	
	it suffices to prove that $q(M^i)=q(\scrM^i)$ where
	$q:\scrD\rightarrow D$ is the canonical projection induced by reduction modulo the ideal $I_+S:= S[1/p] \cap u K_0 [\![u]\!]$.
	Extending $q$ to a map
 $$\wt{q}:B_{\cris}^+\otimes_{S} \scrM =  B_{\cris}^+ \otimes_{S[1/p]} \scrD^i
 \rightarrow W(\o{k})[1/p] \otimes_{K_0} D$$
 in the obvious way, due to Lemma \ref{lem-compatible[1/p]} and our identifications, it is then enough to prove that
 \begin{equation}
 	\wt{q}(A_{\cris}\otimes_{A_{\inf}}\wt{M}^i)=\wt{q}(A_{\cris}\otimes_S \scrM^i),
	\label{qtil}
\end{equation}	
	where
 $A_{\cris}\otimes_{A_{\inf}}\wt{M}^i$ is viewed as a submodule of $B_{\cris}^+\otimes_S \scrM^i$ via (\ref{Lemma:commdiag}).
 But $\wt{\iota}$ carries $A_{\cris}\otimes_{A_{\inf}}\wt{M}^i$ isomorphically on to
 $A_{\cris}\otimes_S \scrM^i$ as we have seen, so the desired equality (\ref{qtil})
 indeed holds.

 Recalling that we have defined $\Fil ^i \scrM ^i := \Fil ^i \scrD ^i \cap \scrM ^i,$
and the map $\iota$ of (\ref{eqn-iso}) is compatible with filtrations,
we find
\begin{equation}
	\Fil^i \scrM^i = \Fil^i \scrD^i \cap \scrM^i = \Fil^i(\u{\scrM}(M^i)[1/p]) \cap \scrM^i
	=\Fil^i(\u{\scrM}(M^i)[1/p]) \cap \u{\scrM}(M^i)
	\label{FilsEq}
\end{equation}
via our identifications. From the construction of $\Fil ^i \u{\scrM}(M^i)$, it is easy to show that $\Fil ^i \u{\scrM}(M^i)$ is saturated as submodule of $\u\scrM (M^i)$
Hence the right side of (\ref{FilsEq})
coincides with $\Fil^i \u{\scrM}(\m^i)$.  We conclude that the $\varphi$-compatible isomorphism
of $S$-modules
$\u{\scrM}(\m^i)\simeq \scrM^i$ induced by $\iota$ is moreover filtration compatible,
from which it follows that $\scrM^i$ is isomorphic to $\u{\scrM}(M^i)$. As $M^i$ is the Breuil--Kisin module corresponding to $(T^i) ^\vee$, Theorem \ref{thm-classical} shows that
$\scrM^i$ is a strongly divisible lattice and $T_{\st} (\scrM) \simeq (T^i)^\vee$.
  By Theorem \ref{MainThm},
we then have a natural isomorphism of Breuil-Kisin modules $\u{M}(\scrM^i)\simeq M^i$
as desired.
\end{proof}

\section{Further directions}

In this section, we discuss some questions
and directions for further research.

\subsection{Geometric interpretation of filtrations}\label{subsec-1}

Let $\X$ be a smooth and proper scheme over $\O_K$,
and let $(M,\Fil^i M,\varphi_{M,i})$ be the
Breuil--Kisin module (in the sense of Definition \ref{def:BKvar})
attached to the dual of the Galois lattice
$T^i:=H^i_{\et}(\X_{\o{K}},\Z_p)/\tors$.
For $p>2$, when $i < p-1$ and $H^j_{\cris}(\X_{k}/W(k))$ is torsion free for $j=i,i+1$,
our main result Theorem \ref{thm-main-crystokisin} provides the canonical ``cohomological" interpretation
\begin{equation*}
	M \simeq \u{M}(H^i_{\cris}(\X_0/S)):=
	\varprojlim_{\varphi,n} \Fil^0(H^i_{\cris}(\X_0 / S) \otimes_{S} S_n[z_n^{-1}]).
\end{equation*}
However, our definition
of the filtration on $\scrM:=H^i_{\cris}(\X_0/S)$---which plays a key role in the very
definition of the Breuil--Kisin module $\u{M}(\scrM)$---is not
as explicitly ``geometric" as one might like.  Indeed, put $V:=T[1/p]$
and denote by $\wt{\scrD}:=\scrD(V^{\vee})$ the filtered $(\varphi,N)$-module over $S[1/p]$
attached to $D:=D_{\cris}(V^{\vee})$ just above Definition \ref{Def-SDlattice}.
Using the Hyodo--Kato isomorphism $\scrM[1/p] \simeq \wt{\scrD}$ of $\varphi$-modules over $S[1/p]$,
we equip $\scrM[1/p]$ with a filtration by ``transport of structure", and have given the crystalline
cohomology $\scrM$ of $\X_0$ the filtration
$\Fil^i \scrM:=\scrM\cap \Fil^i (\scrM[1/p])$; see \S\ref{pdivGone}.

This filtration on $\scrM$ can be defined cohomologiclly as follows.
For $m\ge 0$ set $E_m:=\Spec(S/p^mS)$ and $\Y_m:=\X \times_{\O_K} \O_K/p^m\O_K$, and
let $\J_{m}$ be the sheaf of PD-ideals
on the big crystalline site $\Cris(\Y_m/E_m)$ whose value on the object $(U\hookrightarrow T,\delta)$ is $\ker(\O_T\rightarrow \O_U)$.
Writing $\J_m^{[i]}$ for the $i$-th divided power of $\J_m$,
Tsuji has proved \cite{TsujiNote}  
that one has a canonical
isomorphism
\begin{equation}
	\Fil^i (\scrM[1/p]) \simeq S[1/p]\otimes_S\varprojlim_m H^i((\Y_m/E_m)_{\cris},\J_m^{[i]}).
	\label{filhope}
\end{equation}
One can of course ask if the stronger, $p$-integral version of (\ref{filhope}) holds as well,
that is, whether or not (\ref{filhope}) carries $\Fil^i\scrM$ isomorphically onto
$\varprojlim_m H^i((\Y_m/E_m)_{\cris},\J_m^{[i]}).$
If true, such an isomorphism would of course be the ``best possible"
cohomological interpretation of $\Fil^i\scrM$.

\subsection{Other Frobenius Lifts and Wach Modules}

We expect that our main result can be generalized to give a cohomological
description of the generalization of Breuil--Kisin modules constructed in \cite{Cais-Liu},
which include the Wach modules of Berger \cite{Berger}, \cite{BergerBreuil}, as well as the modules of Kisin--Ren \cite{KisinRen}.
More precisely, let $F\subseteq K$ be a subfield which is finite over $\Q_p$ with residue field $k_F$
of cardinality $q=p^s$ and fixed uniformizer $\varpi$.  Choose a power series
$f(u) := a_1u + a_2u^2 + \cdots \in \O_F[\![u]\!]$
with $f(u)\equiv u^q\bmod \varpi$ and a uniformizer $\pi_0$ of $K$ with minimal polynomial
$E(u)$ over $F_0:=K_0\cdot F$.  Choose $\underline{\pi}:=\{\pi_n\}_{n\ge 1}$
with $\pi_n\in \overline{K}$ satisfying $f(\pi_n)=\pi_{n-1}$ for $n\ge 1$.
The resulting extension $K_{{\underline{\pi}}}:=\bigcup_{n\ge 0} K(\pi_n)$
(called a {\em Frobenius iterate extension} in \cite{CaisDavis}) is an infinite
and totally wildly ramified extension of $K$ which in general need not be Galois,
though in the special case that $v_{\varpi}(a_1)=1$ and $K$ is obtained from $F$
by adjoining the roots of $f(u)=0$, it is a Lubin--Tate extension of $F$.

Define $\s:=W[\![u]\!]$ and put $\s_F   = \O_F \otimes _{W(k_F)} \s$.
We equip $\s_F$ with the (unique continuous)
Frobenius endomorphism $\varphi$ which acts on $W(k)$ by the $q$-power Witt-vector Frobenius,
acts as the identity on $\O_F$,
and sends $u$ to $f(u)$.
Define $S_F$ to be the $\varpi$-adic completion of the $\O_F$-divided power envelope
(in the sense of Faltings \cite{FaltO}) of
the $\O_F$-algebra surjection $\s_F\twoheadrightarrow \O_K$
sending $u$ to $\pi_0$.
There are evident analogues $\Mod_{\s_F}^{\varphi,r}$ and $\Mod_{S_F}^{\varphi,r}$
of the categories of Breuil--Kisin and Breuil modules in this setting, and
the recent Ph.~D. thesis of Henniges
\cite{Henniges} shows that the
canonical base change functor $\Mod_{\s_F}^{\varphi,r}\rightarrow \Mod_{S_F}^{\varphi,r}$
is an isomorphism when $p>2$ and $r < q-1$.  We expect that the methods of the present paper
can be adapted to provide an explicit quasi-inverse to this base change functor,
along the lines of Definition \ref{Def:Mfunctor} and Theorem \ref{MainThm}.
When one moreover has a theory of crystalline cohomology that
produces Breuil modules over $S_F$, we further expect that
Theorem \ref{thm-main-crystokisin} can be generalized, thereby giving
a geometric description of the Breuil--Kisin modules constructed in \cite{Cais-Liu}
or in \cite{KisinRen}.
When $F$ is unramified over $\Q_p$, so $S_F$ is the usual completed PD-envelope of
$\s \twoheadrightarrow \O_K$, then the classical theory of crystalline cohomology already
provides the necessary machinery to carry out this vision.
For general $F$, one also has such a crystalline theory in the Barsotti--Tate setting $r=1$
when $K_{\u{\pi}}/F$ is Lubin--Tate, thanks to the work of Faltings \cite{FaltO}.

Perhaps the simplest and most promising instance of the above framework is when $F=W(k)[1/p]$, $K=F(\mu_{p})$
and $\varphi(u)=(1+u)^p -1$.  One may then choose $\u{\pi}$
so that $K_{\u{\pi}}$ is the cyclotomic extension of $F$.  There is a natural action of
$\Gamma:=\Gal(K_{\u{\pi}}/K)$ on $\s=\s_F$ given by $\gamma u:=(1+u)^{\chi(\gamma)}-1$,
which uniquely extends to $S=S_F$.  We consider categories of modules $\Mod_{\scrS}^{\varphi,\Gamma,r}$
for $\scrS\in \{\s, S\}$ whose objects are Breuil--Kisin or Breuil modules $(M,\Fil^r M,\varphi_{M,r})$ over $\scrS$
that have the additional structure of a semilinear $\Gamma$-action that
is trivial on $M\otimes_{\scrS} W(k)$.
The resulting category $\Mod_{\s}^{\varphi,\Gamma,r}$ of $(\varphi,\Gamma)$-modules over $\s$
is equivalent to the category of {\em Wach modules}, as defined by \cite{Berger} (though
see \cite[\S4.5]{CaisLau} for the claimed equivalences), which classify lattices in crystalline
$G_F$-representations.
We are confident that the main results of the present paper can be readily adapted
to the above setting, thus giving a cohomological interpretation of Wach modules, at least
in Hodge--Tate weights at most $p-2$.

\subsection{Generalization to semistable schemes}

It is natural to ask to what extent the results of this paper
can be generalized to the case of {\em semistable reduction}, that is,
regular proper and flat schemes $\X$
over $\O_K$ with special fiber $\X_k$ that is a reduced normal corssings divisor
on $\X$.  It seems reasonable to guess that the analogue of Theorem \ref{thm-main-crystokisin}
using log-crystalline cohomology continues to hold.
In the case of low ramification $ei < p-1 $,
it should be straightforward to prove
that this is indeed the case using work of Caruso \cite{CarusoInvent}
(generalizing earlier work of Breuil \cite{Breuil} in the case $e=1$),
which provides the essential integral comparison isomorphisms
needed to adapt the arguments of \S\ref{pdivGone} to this setting.
To get results without restriction on the ramification
of $K$ requires the generalization of \cite[Theorem 1.8]{SBM}
to the case of semistable reduction,
which has recently been established by Cesnavicius and Koshikawa \cite{AinfSST}.

\bibliography{mybib}

\def\cprime{$'$}
\begin{thebibliography}{10}

\bibitem{Bar}
Nicol{\'a}s B{\"a}r.
\newblock Towards the cohomological construction of {B}reuil--{K}isin modules.
\newblock \url{https://nojb.github.io/files/thesis.pdf}, 2012.
\newblock Thesis (Ph.D.)--The University of Chicago.

\bibitem{Beilin}
A.~Beilinson.
\newblock On the crystalline period map.
\newblock {\em Camb. J. Math.}, 1(1):1--51, 2013.

\bibitem{Berger}
Laurent Berger.
\newblock Limites de repr\'esentations cristallines.
\newblock {\em Compos. Math.}, 140(6):1473--1498, 2004.

\bibitem{BergerBreuil}
Laurent Berger and Christophe Breuil.
\newblock Sur quelques repr\'esentations potentiellement cristallines de {${\rm
  GL}_2(\mathbf Q_p)$}.
\newblock {\em Ast\'erisque}, (330):155--211, 2010.

\bibitem{crystal}
Pierre Berthelot and Arthur Ogus.
\newblock {\em Notes on crystalline cohomology}.
\newblock Princeton University Press, Princeton, N.J., 1978.

\bibitem{SBM}
Bhargav Bhatt, Matthew Morrow, and Peter Scholze.
\newblock Integral $p$-adic {H}odge theory.
\newblock {\em Publ. Math. Inst. Hautes \'{E}tudes Sci.}, 128: 219--397, 2018.


\bibitem{Breuil-Griffith}
Christophe Breuil.
\newblock Repr\'esentations {$p$}-adiques semi-stables et transversalit\'e de
  {G}riffiths.
\newblock {\em Math. Ann.}, 307(2):191--224, 1997.

\bibitem{Breuil-SDLattice}
Christophe Breuil.
\newblock Repr\'esentations semi-stables et modules fortement divisibles.
\newblock {\em Invent. Math.}, 136(1):89--122, 1999.

\bibitem{Breuil-normes}
Christophe Breuil.
\newblock Une application de corps des normes.
\newblock {\em Compositio Math.}, 117(2):189--203, 1999.

\bibitem{Breuil}
Christophe Breuil.
\newblock Groupes {$p$}-divisibles, groupes finis et modules filtr\'es.
\newblock {\em Ann. of Math. (2)}, 152(2):489--549, 2000.

\bibitem{BreuilIntegral}
Christophe Breuil.
\newblock Integral {$p$}-adic {H}odge theory.
\newblock In {\em Algebraic geometry 2000, {A}zumino ({H}otaka)}, volume~36 of
  {\em Adv. Stud. Pure Math.}, pages 51--80. Math. Soc. Japan, Tokyo, 2002.

\bibitem{CaisLau}
Bryden Cais and Eike Lau.
\newblock {Dieudonn\'{e} crystals and Wach modules for p-divisible groups}.
\newblock {\em Proc. Lond. Math. Soc.}, 114 (3): 733--763, 2017.

\bibitem{CaisDavis}
Bryden Cais and Christopher Davis.
\newblock Canonical cohen rings for norm fields.
\newblock {\em  Int. Math. Res. Not. IMRN}, 14: 5473--5517, 2015.


\bibitem{Cais-Liu}
Bryden Cais and Tong Liu.
\newblock On {$F$}-crystalline representations.
\newblock {\em Doc. Math.}, 21:223--270, 2016.

\bibitem{CarusoInvent}
Xavier Caruso.
\newblock Conjecture de l'inertie mod\'er\'ee de {S}erre.
\newblock {\em Invent. Math.}, 171(3):629--699, 2008.

\bibitem{CarusoLiu}
Xavier Caruso and Tong Liu.
\newblock Quasi-semi-stable representations.
\newblock {\em Bull. Soc. Math. France}, 137(2):185--223, 2009.

\bibitem{AinfSST}
K\k{e}stutis \v{C}esnavi\v{c}ius and Teruhisa Koshikawa.
\newblock The $A_{\inf}$-cohomology in the semistable case.
\newblock {\em arXiv:1710.06145}, 2017.


\bibitem{Faltings}
Gerd Faltings.
\newblock Integral crystalline cohomology over very ramified valuation rings.
\newblock {\em J. Amer. Math. Soc.}, 12(1):117--144, 1999.

\bibitem{FaltO}
Gerd Faltings.
\newblock Group schemes with strict {$\mathscr{O}$}-action.
\newblock {\em Mosc. Math. J.}, 2(2):249--279, 2002.
\newblock Dedicated to Yuri I. Manin on the occasion of his 65th birthday.

\bibitem{FontaineMessing}
Jean-Marc Fontaine and William Messing.
\newblock {$p$}-adic periods and {$p$}-adic {\'e}tale cohomology.
\newblock In {\em Current trends in arithmetical algebraic geometry ({A}rcata,
  {C}alif., 1985)}, volume~67 of {\em Contemp. Math.}, pages 179--207. Amer.
  Math. Soc., Providence, RI, 1987.

\bibitem{Henniges}
Alex~Jay Henniges.
\newblock {\em Kisin-{R}en classification of {$\overline\omega$}-divisible
  {O}-modules via the {D}ieudonne {C}rystal}.
\newblock ProQuest LLC, Ann Arbor, MI, 2016.
\newblock Thesis (Ph.D.)--The University of Arizona.

\bibitem{Hyodo-Kato}
Osamu Hyodo and Kazuya Kato.
\newblock Semi-stable reduction and crystalline cohomology with logarithmic
  poles.
\newblock {\em Ast\'erisque}, (223):221--268, 1994.
\newblock P{\'e}riodes $p$-adiques (Bures-sur-Yvette, 1988).

\bibitem{KisinFcrystal}
Mark Kisin.
\newblock Crystalline representations and {$F$}-crystals.
\newblock In {\em Algebraic geometry and number theory}, volume 253 of {\em
  Progr. Math.}, pages 459--496. Birkh\"auser Boston, Boston, MA, 2006.

\bibitem{Kisin-Modularity}
Mark Kisin.
\newblock Moduli of finite flat group schemes, and modularity.
\newblock {\em Ann. of Math. (2)}, 170(3):1085--1180, 2009.

\bibitem{KisinRen}
Mark Kisin and Wei Ren.
\newblock Galois representations and {L}ubin-{T}ate groups.
\newblock {\em Doc. Math.}, 14:441--461, 2009.

\bibitem{Lau:Frames}
Eike Lau.
\newblock Frames and finite group schemes over complete regular local rings.
\newblock {\em Doc. Math.}, 15:545--569, 2010.

\bibitem{liu-Fontaine}
Tong Liu.
\newblock Torsion {$p$}-adic {G}alois representations and a conjecture of
  {F}ontaine.
\newblock {\em Ann. Sci. \'Ecole Norm. Sup. (4)}, 40(4):633--674, 2007.

\bibitem{LiuT-CofBreuil}
Tong Liu.
\newblock On lattices in semi-stable representations: a proof of a conjecture
  of {B}reuil.
\newblock {\em Compos. Math.}, 144(1):61--88, 2008.

\bibitem{ScholzeWeinstein}
Peter Scholze and Jared Weinstein.
\newblock {$p$}-adic geometry. lecture notes from course at uc berkeley in fall
  2014.
\newblock \url{http://math.bu.edu/people/jsweinst/Math274/ScholzeLectures.pdf}.

\bibitem{stacks-project}
The {Stacks Project Authors}.
\newblock {\itshape Stacks Project}.
\newblock \url{http://stacks.math.columbia.edu}, 2016.

\bibitem{Tsuji}
Takeshi Tsuji.
\newblock {$p$}-adic \'etale cohomology and crystalline cohomology in the
  semi-stable reduction case.
\newblock {\em Invent. Math.}, 137(2):233--411, 1999.

\bibitem{TsujiNote}
Takeshi Tsuji. 
\newblock On the {H}odge filtration of crystalline cohomology over a ramified base.
\newblock {\em email to B. Cais}, May 24, 2017.


\bibitem{VZ}
Adrian Vasiu and Thomas Zink.
\newblock Breuil's classification of {$p$}-divisible groups over regular local
  rings of arbitrary dimension.
\newblock In {\em Algebraic and arithmetic structures of moduli spaces
  ({S}apporo 2007)}, volume~58 of {\em Adv. Stud. Pure Math.}, pages 461--479.
  Math. Soc. Japan, Tokyo, 2010.

\end{thebibliography}

\end{document}